%% file: journal-immersed.tex
\documentclass[11pt]{article}
\pdfoutput=1
\usepackage[latin1]{inputenc}
\usepackage{amsmath,amssymb}
\usepackage[a4paper, margin=1in]{geometry}
\usepackage{algorithmic}
\usepackage{graphicx,url,amsmath,amsthm,amsfonts,amssymb,subfigure,bbm,bm}
\usepackage{color}
\usepackage{verbatim}
\usepackage{wrapfig}
\usepackage{cite}
\newcommand{\executeiffilenewer}[3]{%
\ifnum\pdfstrcmp{\pdffilemoddate{#1}}%
{\pdffilemoddate{#2}}>0%
{\immediate\write18{#3}}\fi%
}

\definecolor{blueblack}{rgb}{0,0,.7}

\newcommand{\R}{\mathbb{R}}
\newcommand{\Z}{\mathbb{Z}}

\newcommand{\CNF}{\textrm{CNF}}
\newcommand{\SAT}{\textrm{SAT}}

\newtheorem{theorem}{Theorem}
\newtheorem{lemma}[theorem]{Lemma}

\newtheorem{corollary}[theorem]{Corollary}

\newtheorem{proposition}[theorem]{Proposition}
\newtheorem{problem}[theorem]{Problem}

\title{On the Complexity of Immersed Normal Surfaces}
\author{Benjamin A.~Burton%                                                     
  \thanks{University of Queensland, Brisbane, Australia. email:                 
    \protect\url{bab@maths.uq.edu.au}. Supported by the Australian Research     
    Council (projects DP1094516 and DP110101104). Portions of this work         
    were done while this author was invited professor at \'Ecole normale        
    sup\'erieure.} \and \'Eric Colin de Verdi\`ere%                             
  \thanks{D\'epartement d'informatique, \'Ecole normale sup\'erieure, CNRS,     
    Paris, France. email: \protect\url{Eric.Colin.de.Verdiere@ens.fr}.          
    Portions of this work were done while this author was invited professor     
    at University of Queensland.} %                                             
  \and Arnaud de Mesmay%                                                        
  \thanks{D\'epartement d'informatique, \'Ecole normale sup\'erieure,           
    Paris, France. email: \protect\url{Arnaud.de.Mesmay@ens.fr}}}%              
\date{\today}

\begin{document}

\maketitle\thispagestyle{empty}

\begin{abstract}
  Normal surface theory, a tool to represent surfaces in a triangulated
  3-manifold combinatorially, is ubiquitous in computational 3-manifold
  theory. In this paper, we investigate a relaxed notion of normal surfaces
  where we remove the quadrilateral conditions. This yields normal surfaces
  that are no longer embedded. We prove that it is NP-hard to decide
  whether such a surface is immersed. Our proof uses a reduction from
  Boolean constraint satisfaction problems where every variable appears in
  at most two clauses, using a classification theorem of Feder. We also
  investigate variants, and provide a polynomial-time algorithm to test for
  a local version of this problem.
\end{abstract}

%%%%%%%%%%%%%%%%%%%%%%%%%%%%%%%%%%%%%%%%%%%%%%%%%%%%%%%%%%%%%%%%%%%%%%%%%%
\section{Introduction}

The field of computational topology aims at providing computational and
efficient tools to deal with topological problems. In this theory, the
dimension of the problems considered has a very direct impact on the
complexity of the algorithms designed to solve them. Fundamental problems
tend to have polynomial-time solutions for surfaces~\cite{c-tags-12}, while
in dimensions four and larger, things easily become
undecidable~\cite[Chapter 9]{s-ctcgt-93}. In the intermediate case, most of
the problems encountered in 3-dimensional topology are decidable but
typically solved with (at least) exponential-time algorithms. Probably the
most famous of these is unknot recognition, whose complexity is in
NP~\cite{hlp-ccklp-99}, and in co-NP assuming the Generalized Riemann
Hypothesis~\cite{k-knmg-11l}, but for which no polynomial-time algorithm
nor hardness proof is known.

\paragraph{Normal surfaces.}
A standard way to study 3-manifolds is to investigate which surfaces can be
embedded in them. Normal surfaces, used in a wealth of algorithms, are
perhaps the most ubiquitous tool for this purpose. First brought to the
algorithmic light by Haken~\cite{h-tnik-61}, normal surfaces provide a
compact and structured way to analyze and enumerate the most interesting
surfaces embedded in a 3-manifold. Starting with a triangulation $T$ of a
3-manifold $M$ with $t$ tetrahedra, a \emph{normal surface} is a (possibly
disconnected) surface in~$M$ whose intersection with each tetrahedron is a
disjoint union of disks of seven simple possible types (see
Figure~\ref{F:Normaldisks} and Section~\ref{S:normsurf} below). This allows
a normal surface to be entirely described by a vector in
$\mathbb{Z}_+^{7t}$, its \emph{normal coordinates}.  Many interesting
surfaces, such as for example a Seifert disk of an unknotted closed curve,
are witnessed by a normal surface having coordinates at most exponential
in~$t$; see Hass, Lagarias, and Pippenger~\cite{hlp-ccklp-99}. This is the
starting point of many algorithms based on the enumeration of normal
surfaces, which naturally have an exponential complexity.

In addition to providing a succinct representation of embedded surfaces,
normal surfaces also possess an additional algebraic structure. Indeed, the
natural addition and scalar multiplication of vectors translate to
operations on normal surfaces, and the space of normal surfaces
in~$\mathbb{R}^{7t}$ is characterized by a set of equations: the
\emph{matching equations} and the \emph{quadrilateral conditions}. The
former are linear equations specifying the way to glue normal surfaces
locally, while the latter are non-linear and ensure that the resulting
surface is embedded. Spaces defined by linear constraints can be studied by
the means of linear programming~\cite{mg-uulp-07}, which provides a very
powerful framework to deal with decision and optimization problems. This
motivates the study of a notion of relaxed normal surfaces, where we remove
the quadrilateral conditions to obtain a simpler, polyhedral structure on
the space of normal surfaces.

\paragraph{Our results.}
As we shall see later, removing the quadrilateral conditions amounts to
removing the embeddedness of normal surfaces. Therefore, it amounts to
dealing with \emph{singular normal surfaces}. Among these, the
\emph{immersed normal surfaces} are well behaved, in the sense that while
they can self intersect, they are still 2-manifolds locally.  Moreover,
their Euler characteristic depends linearly on their normal
coordinates---this fact is crucial in algorithms that work with embedded
normal surfaces~\cite{hlp-ccklp-99}, but does not hold for all singular
normal surfaces.  By coupling singular normal surface theory with an
algorithm that efficiently separates immersed normal surfaces from the
others, we would have powerful tools at our disposal: this could lead to
efficient algorithms to find immersed low genus surfaces in
3-manifolds. Furthermore, through classical topological results like Dehn's
lemma or the loop theorem \cite{h-nb3mt-07} we would obtain embedded
surfaces, which are the key behind the unknot problem and many others.

In this article, we show some inherent limitations of this method by
proving in Theorem~\ref{T:main} that it is NP-hard to detect whether a
given set of normal coordinates can represent an immersed normal surface.  In other words, there exists no polynomial-time algorithm for
this problem unless P=NP.

Immersed normal surfaces have been studied from a mathematical point of
view by Letscher~\cite{l-insdp-97} and, in the particular case of the
figure-eight knot complement, from a computational perspective by
Aitchison, Matsumoto, and Rubinstein~\cite{amr-sf8kc-98},
Rannard~\cite{r-cinsf-99}, and Matsumoto and Rannard~\cite{mr-rpssf-00}.
In the latter papers, the authors devise and implement an algorithm to
decide whether a given set of normal coordinates can represent an immersed normal surface, and find an such an immersed surface if there is one.  While the
complexity of this algorithm is not explicitly computed, it is at least
doubly exponential in the input size.  Our main result shows that the
problem is inherently hard and that no polynomial-time solution is to be
expected.

The complexity reduction used in the proof of this theorem works by
reducing the problem to a satisfiability problem, which at first sight appears to follow a standard technique. However, the flavor of this reduction
is that in this problem, it turns out to be very hard to obtain more
than two copies of every variable. Our proof thus relies on relatively
intricate classification theorems on the complexity of Boolean
constraint satisfaction problems where every variable occurs at most
twice~\cite{df-gslov-03, f-flcs-01}.  This approach is thus, to some
extent, original, and might prove useful to obtain other hardness
proofs that could be hard to achieve by other traditional means.

Hardness results are scarce in 3-dimensional computational topology, and to
our knowledge all the other difficulty results are deduced from the Agol,
Hass, and Thurston construction~\cite{aht-cckgs-06}, except for the recent
hardness results on computing taut angle structures\cite{bs-cdtast-13} and
optimal Morse matchings~\cite{blps-pcdmt-13}. Our result displays a different
intractability aspect of this theory.  In contrast to the aforementioned
result by Agol, Hass, and Thurston~\cite{aht-cckgs-06}, we also prove that
this problem is NP-hard even when the input triangulation is a sub-manifold
of $\mathbb{R}^3$, which is for example the case for the very important
class of knot complements.

On the upper bound side, it is a natural question to wonder whether this
problem can be solved in polynomial time when the size of the triangulation
is fixed. As a partial evidence for this, the aforementioned work of
Matsumoto and Rannard~\cite{mr-rpssf-00} suggests that the problem may be
solved in polynomial time in the specific case of the figure-eight knot
complement. Although we make no progress on this question in full
generality, we show that if the triangulation just consists of tetrahedra
all sharing a single edge, the problem has a polynomial-time solution as it
can solved by computing a maximum flow.  Another view on this algorithm is
that it can certify \emph{local} immersibility, where the locality means
that it can only check whether there is an obstruction to being immersed
around every edge, but it is not global since ensuring immersibility as
some point might force a branch point at some other point for the
triangulation.

This paper is organized as follows. We start by introducing the main
concepts of $3$-manifold topology, normal surface theory, and Boolean
constraint satisfaction in Section~\ref{S:prelim}. In Section~\ref{S:thm},
we describe our reduction and prove the main
theorem. Section~\ref{S:variants} explores some variants of the
immersibility problem that remain NP-hard. Finally, in
Section~\ref{S:local}, we provide an algorithm to test immersibility in the
restricted case of tetrahedra glued around a single edge.

%%%%%%%%%%%%%%%%%%%%%%%%%%%%%%%%%%%%%%%%%%%%%%%%%%%%%%%%%%%%%%%%%%%%%%%%%%
\section{Preliminaries}\label{S:prelim}

\subsection{Background on normal surface theory}\label{S:normsurf}

For completeness, we first give a (rather standard) background on normal
surface theory; see e.g., Hass, Lagarias, and Pippenger~\cite{hlp-ccklp-99}
or Gordon~\cite{g-tns-01} for more details.  A \emph{3-manifold with boundary}
is a compact topological space such that every point is locally
homeomorphic to $\R^3$ or to the closed half space $\{x,y,z \mid x,y,z \in
\R, z \geq 0\}$. We will describe 3-manifolds using \emph{triangulations}:
A triangulation~$T$ is a topological space obtained from a disjoint set of
$t$~tetrahedra $T=(T_1,\ldots,T_t)$ by (combinatorially) gluing some pairs
of two-dimensional faces of these tetrahedra; a gluing between two faces is
specified by a bijection from the vertex set of the first face to the
vertex set of the second face.  As a result of these gluing, edges and
vertices of tetrahedra are also identified; it is also allowed to glue two
zero-, one-, or two-dimensional faces of the same tetrahedron.  A
\emph{face} of a triangulation~$T$ is a two-dimensional simplex, incident
to one or two tetrahedra in~$T$.

% The \emph{link} of a vertex~$x$ in a triangulation $T$ is the set of
% faces $f$ such that the tetrahedron corresponding to $f$ and $x$ is
% included in $T$.

The following lemma is standard~\cite{m-as3mtth-52}; we include it for
further reference:
\begin{lemma}\label{L:manifold}
  A triangulation is a 3-manifold (possibly with boundary) if and only if:
  \begin{enumerate}
  \item each vertex has a neighborhood homeomorphic to~$\R^3$ or to the
    closed half-space;
  \item after the gluings, no edge is identified to itself in the reverse
    orientation.
\end{enumerate}
\end{lemma}
(These two conditions are obviously necessary, and turn out to be
sufficient to ensure that every point of the triangulation has a manifold
neighborhood.)

Henceforth, $T$ denotes a triangulation of a 3-manifold~$M$.  A \emph{normal
  isotopy} is an ambient isotopy of $M$ that is fixed on the $2$-skeleton
of $T$.  A \emph{normal surface} in~$T$ is a properly embedded surface
in~$T$ that meets each tetrahedron in a (possibly empty) disjoint
collection of \emph{normal disks}, each of which being either a
\emph{triangle} (separating one vertex of the tetrahedron from the other
three) or a \emph{quadrilateral} (separating two vertices from the other
two).  In each tetrahedron, there are $4$ possible types of triangles and
$3$ possible types of quadrilaterals, pictured in
Figure~\ref{F:Normaldisks}.  The intersection of an embedded normal surface
with a face of the triangulation gives rise to a \emph{normal arc}.  There
are $3$ possible types of normal arcs within each face: the type of a
normal arc is defined according to which vertex of the face it separates
from the other two.

\begin{figure}[htb]
\centering
\def\svgwidth{13cm}
\executeiffilenewer{fig/Normaldisks.svg}{fig/Normaldisks.pdf}%
{inkscape -z -D --file=fig/Normaldisks.svg %
--export-pdf=fig/Normaldisks.pdf --export-latex}%
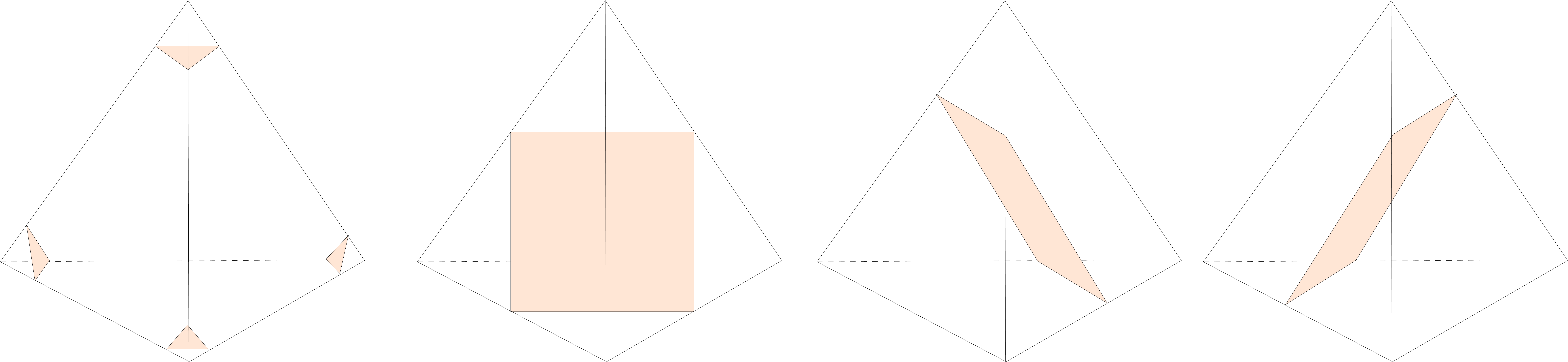%

\caption{The seven types of normal disks within a given tetrahedron: Four
  triangles and three quadrilaterals.}
\label{F:Normaldisks}
\end{figure}

Each embedded normal surface has associated \emph{normal coordinates}: a
vector in $(\Z_+)^{7t}$, where $t$ is the number of tetrahedra in $T$,
listing the number of triangles and quadrilaterals of each type in each
tetrahedron.  The normal coordinates provide a very compact and elegant
description of that surface, and satisfy two types of conditions:
\begin{itemize}
\item The first type of conditions is the \emph{matching equations}.
  Consider a normal arc type in a given non-boundary face~$f$ of~$T$.  This
  normal arc type corresponds to exactly one triangle normal coordinate,
  $v_{t,1}$, and one quadrilateral normal coordinate, $v_{q,1}$, in a
  tetrahedron incident with~$f$.  Similarly, let $v_{t,2}$ and~$v_{q,2}$ be
  the triangle and quadrilateral normal coordinates corresponding to that
  arc type in the opposite tetrahedron.  The matching equation for that arc
  type is, by definition, $v_{t,1}+v_{q,1}=v_{t,2}+v_{q,2}$.
\item The second type of conditions, the \emph{quadrilateral
    conditions}, stipulates that, within any tetrahedron, at most one of
  the three quadrilateral coordinates must be non-zero.  Indeed, two
  quadrilaterals of different types within the same tetrahedron must cross,
  and therefore this condition is needed to ensure that the surface does
  not self-intersect.
\end{itemize}

Conversely, if $T$ is a triangulation of size $t$ and $v \in
(\Z_+)^{7t}$. Then $v$ corresponds to an embedded normal surface if and
only if the matching equations and the quadrilateral conditions are
fulfilled. The reconstruction process can be described as follows:

\begin{itemize}
\item In each tetrahedron, by the quadrilateral conditions, there is at
  most one non-zero quadrilateral coordinate.  One places as many parallel
  copies of this quadrilateral as needed in the tetrahedron, and then place
  each triangle close to the vertex of the tetrahedron that it separates
  from the other three.  It is straightforward to do so without having any
  intersection between triangles and quadrilaterals.
\item One glues the faces on the triangulation together, and in the
  process, one needs to glue normal arcs, i.e., triangles or quadrilaterals
  on the one side to triangles and quadrilaterals on the other side. By the
  matching equations, the numbers fit, and the gluing is imposed by the
  order in which the normal disks are placed in the tetrahedra.
\end{itemize}

Therefore, an embedded normal surface is represented up to a normal isotopy
by a vector in $(\Z_+)^{7t}$ satisfying the matching equations and the
quadrilateral conditions.  Moreover, given a triangulation and normal
coordinates, checking that the matching equations or the quadrilateral
conditions hold can trivially be done in linear time.

From this construction, one sees moreover that every set of normal coordinates
corresponds to a unique normal surface, up to a normal isotopy.

\subsection{Singular and immersed normal surfaces}

Consider a vector~$v\in(\Z_+)^{7t}$ of normal coordinates satisfying the
matching equations, but not necessarily the quadrilateral conditions.  For
each $i\in\{1,\ldots,7t\}$, build $v_i$~normal disks of the corresponding
type in the corresponding tetrahedron, in general position, in a way that,
on each non-boundary face~$f$ of~$T$, the images of the normal arcs arising
from both sides of~$f$ agree.  The matching equations imply that such a
construction is always possible.  Note that, with this gluing, we do not
forbid intersections between normal disks in a single tetrahedron, and that
such intersections are actually necessary if two different quadrilateral
coordinates within the same tetrahedron are non-zero.

More precisely, consider a given normal arc type in a given non-boundary
face~$f$ of~$T$, corresponding (as above) to two normal coordinates
$v_{t,1}$ and~$v_{q,1}$ in a tetrahedron incident to~$f$, and also to two
normal coordinates $v_{t,2}$ and~$v_{q,2}$ in the adjacent tetrahedron.
Recall that the matching equations imply that $v_{t,1}+v_{q,1}$ and
$v_{t,2}+v_{q,2}$ are equal.  The data of a bijection between these
$v_{t,1}+v_{q,1}$ normal disks in the first tetrahedron with these
$v_{t,2}+v_{q,2}$ normal disks in the second tetrahedron is called the
\emph{local gluing} of that arc type.  The aggregated information of all
the local gluings is called the \emph{global gluing}.  

\begin{figure}[htb]
\centering
\def\svgwidth{4.5cm}
\executeiffilenewer{fig/singular.svg}{fig/singular.pdf}%
{inkscape -z -D --file=fig/singular.svg %
--export-pdf=fig/singular.pdf --export-latex}%
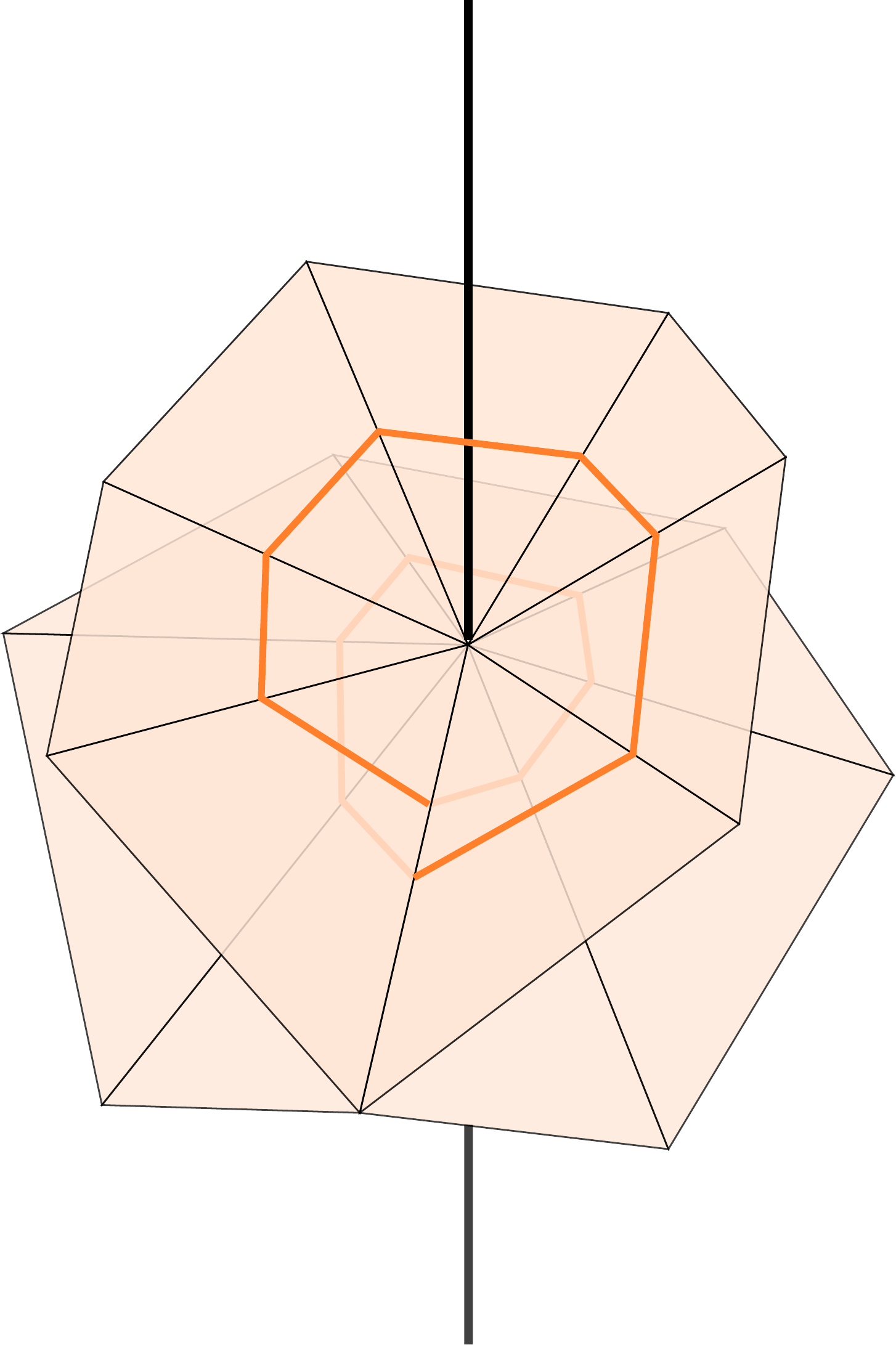%

\caption{A branch point of a singular normal surface.}
\label{F:singular}
\end{figure}

The union of these normal disks glued according to such rules is the image
of a surface under a continuous map, since abstractly gluing triangles and
quadrilaterals by pairwise identifications of edges always results in a
surface (whose actual geometric realization in~$T$ may self-cross).  This
is called a \emph{singular normal surface}.  The continuous map may either
be locally one-to-one, in which case it is called an \emph{immersed normal
  surface}, or have \emph{branch points}, as pictured in
Figure~\ref{F:singular}.  Since normal disks are embedded within each
tetrahedron, any branch point of a singular surface is necessarily on an
edge of the triangulation; it corresponds to a small closed curve on the
surface ``winding more than once'' around the edge of the triangulation.

\begin{figure}[htb]
\centering
\def\svgwidth{10cm}
\executeiffilenewer{fig/abstraction.svg}{fig/abstraction.pdf}%
{inkscape -z -D --file=fig/abstraction.svg %
--export-pdf=fig/abstraction.pdf --export-latex}%
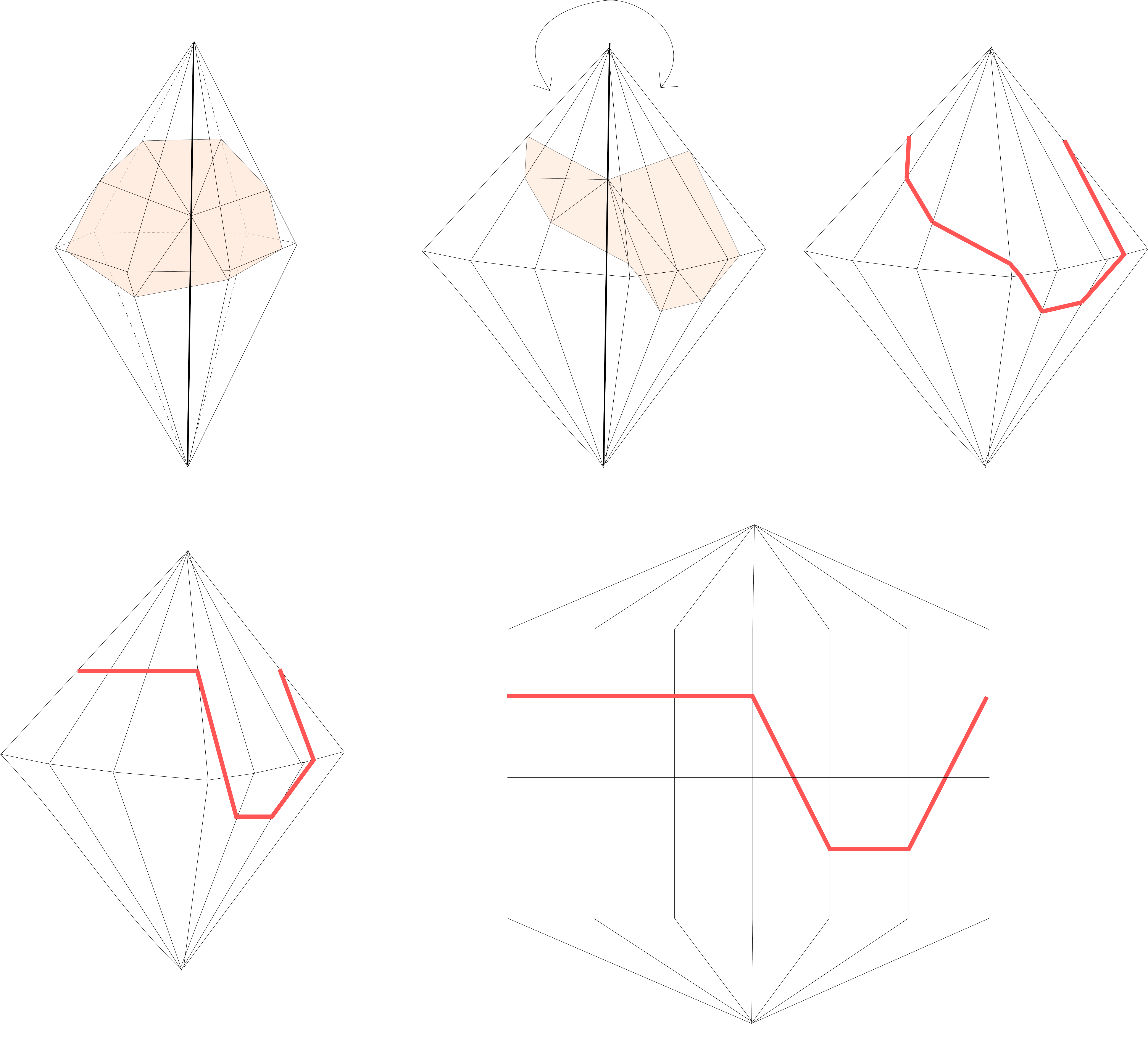%

\caption{A \emph{block}, namely, a set of tetrahedra sharing a common
  non-boundary edge, depicted with a part of a singular normal surface
  inside it, and its abstract representation.}
\label{F:abstraction}
\end{figure}

\subsection{The immersibility problem}

The main problem studied in this paper is the following.  Consider now some
normal coordinates satisfying the matching equations, but not necessarily
the quadrilateral conditions.  As described above, there are many ways to
glue the corresponding normal disks to obtain a singular normal surface.
Then, depending on the choice of the gluings, some of the resulting
surfaces may be immersed while some other may have branch points.  If there
exists a global gluing whose corresponding singular surface is immersed, we
say that the normal coordinates are \emph{immersible}. In this chapter, we
study the computational complexity of the following problem:
\begin{problem}[\textsc{Immersibility}]\ \\
  \textbf{Input}: A triangulation $T$ and normal coordinates $N$.\\
  \textbf{Output}: Are the normal coordinates $N$ immersible?
\end{problem}
Two difficulties lie at the heart of this problem: Not only do we need to
guess a ``good'' gluing, but this gluing may have an exponential complexity
in the input, since the normal coordinates are naturally compressed by the
bit representation. Therefore, the naive algorithm (implemented by
Matsumoto and Rannard~\cite{mr-rpssf-00}) is doubly exponential, which can
be seen just at the level of local gluings: If there are $k$~arcs of a given
arc type, there are $k!$ possible local gluings, while the normal
coordinates involved are $O(k)$ and can thus be encoded in $O(\log k)$
bits.

\subsection{Representation of singular normal surfaces}

To describe the singular normal surfaces more accurately, we now introduce
a schematic representation of singular normal surfaces.

A \emph{block} is a family of tetrahedra that all have one edge~$e$ in
common; see Figure~\ref{F:abstraction}(a); we always assume that $e$~is not
on the boundary of~$T$.  In what follows, we consider the normal disks that
intersect~$e$.  Since we want to picture cleanly what happens on the back
of this block, we will unfold it as in Figure~\ref{F:abstraction}(b), with
the implicit convention that the rightmost face is glued to the leftmost
face.  Although normal disks can be drawn inside this block, the pictures
easily become congested when there are several of them.  Instead, we will
forget the edge~$e$ in the representation and represent the normal disks
intersecting~$e$ by their normal arcs, i.e., by their intersection with the
front faces (Figure~\ref{F:abstraction}(c)).  These normal arcs are glued
together and form possibly self-intersecting closed curves, called
\emph{block curves}.  Abstracting a bit more, horizontal lines will
represent triangles, while diagonal ones will stand for quadrilaterals
(Figure~\ref{F:abstraction}(d)).  Finally, to make these pictures even more
readable, we will draw the edges between the tetrahedra vertically, only
linking them at the extreme top and bottom parts of the figures
(Figure~\ref{F:abstraction}(e)).

\begin{figure}[htb]
\centering
\def\svgwidth{11cm}
\executeiffilenewer{fig/usual.svg}{fig/usual.pdf}%
{inkscape -z -D --file=fig/usual.svg %
--export-pdf=fig/usual.pdf --export-latex}%
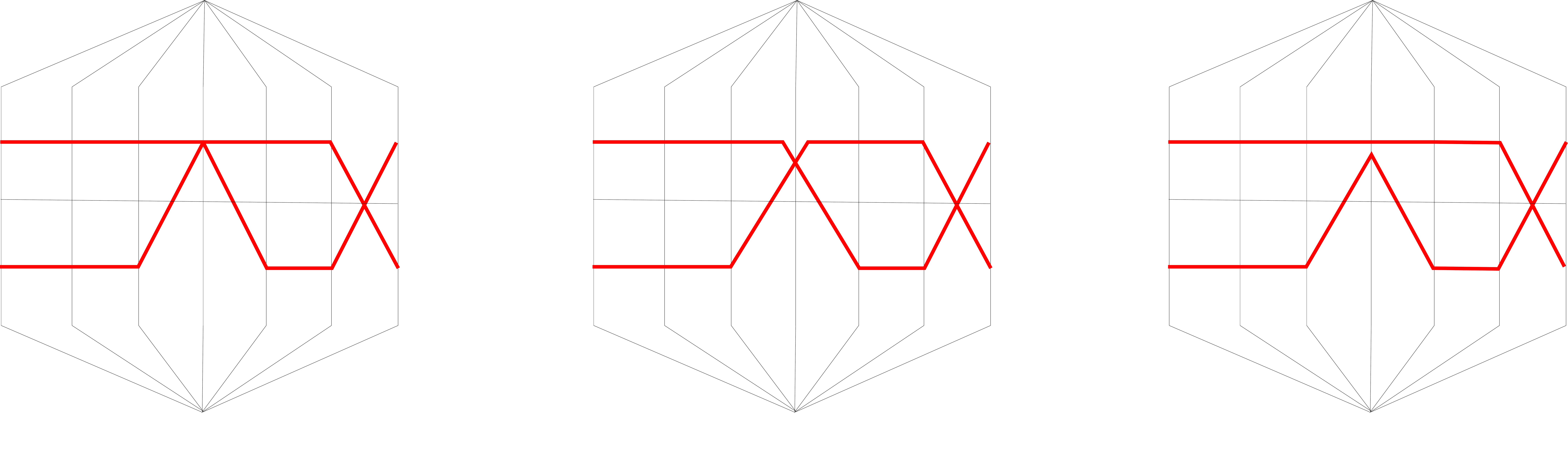%

\caption{Normal coordinates drawn (a) without specifying gluings, (b) with
  a specified gluing, (c) with the opposite gluing from (b).}
\label{F:usual}
\end{figure}

We will use the following convention in the figures (see
Figure~\ref{F:usual}): Whenever we want to represent normal coordinates,
without a specific gluing, the normal arcs are drawn so that they connect
the midpoints of the corresponding edges of the triangulation.  Whenever we
want to represent normal coordinates with a particular gluing, we perturb
these normal arcs to emphasize the crossings.

The singular normal surface has a branch point at~$e$ if and only if some
block curve ``winds more than once'' around~$e$ or, equivalently,
self-intersects, like in Figure~\ref{F:branch}.  With this in mind, it is
easy to see that a branch point can only occur on a \emph{non-boundary}
edge of~$T$.

\begin{figure}[htb]
\centering
\def\svgwidth{12cm}
\executeiffilenewer{fig/branch.svg}{fig/branch.pdf}%
{inkscape -z -D --file=fig/branch.svg %
--export-pdf=fig/branch.pdf --export-latex}%
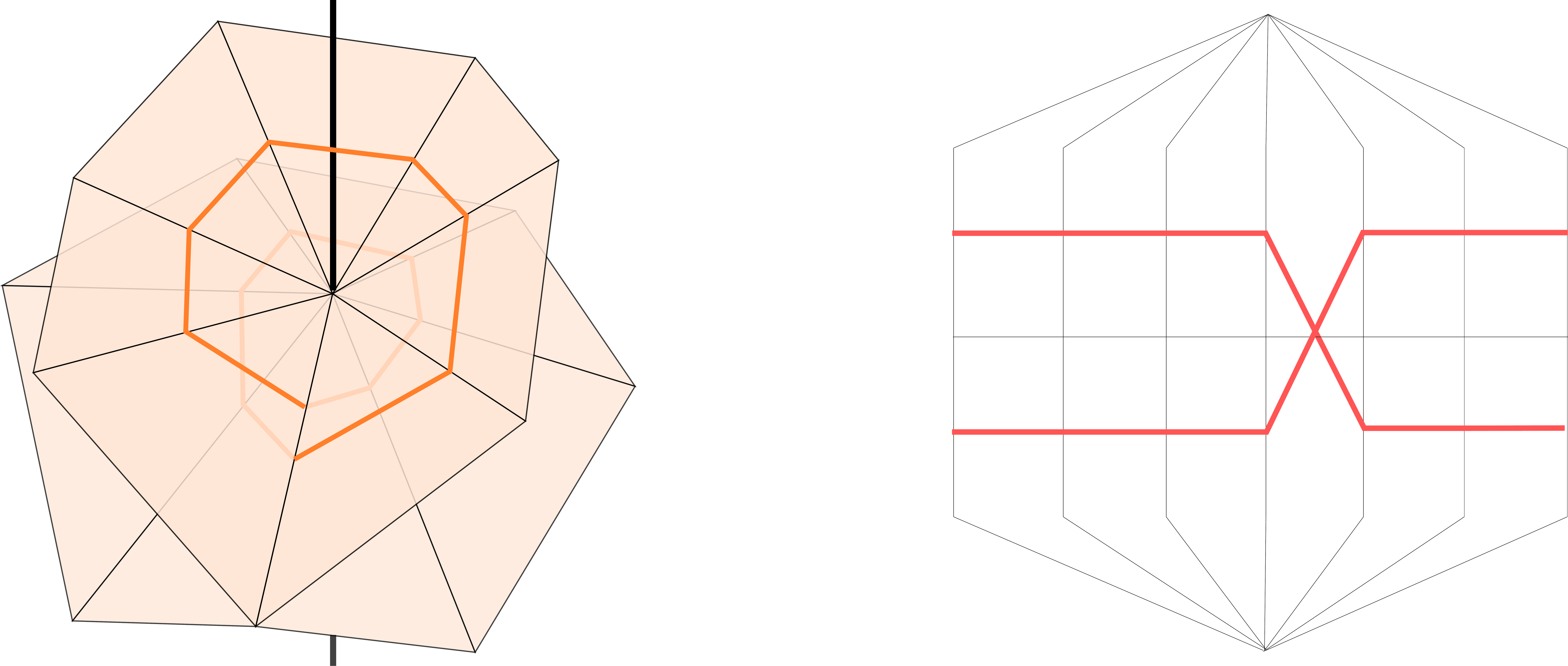%

\caption{A branch point (also depicted in Figure~\ref{F:singular}) and
  its representation by a block curve winding twice around an edge.}
\label{F:branch}
\end{figure}

Therefore, given the data of normal coordinates satisfying the matching
equations, together with a global gluing, we can easily determine whether
the corresponding singular normal surface is immersed or not, in time
linear in the size of the input (namely, sum of all the normal coordinates,
of the complexity of the triangulation~$T$ and the complexity of the global
gluing).  However, our goal is to prove that, given only the normal
coordinates, it is NP-hard to decide whether some global gluing leads to an
immersed surface.

\subsection{Complexity classes}

We now very briefly introduce some standard notions of complexity theory;
see, e.g., Arora and Barak~\cite{ab-ccma-06} for more details.  P denotes
the class of yes/no problems solvable in polynomial time.  NP denotes the
class of yes/no problems whose ``yes'' instances admit a polynomial-sized
\emph{certificate}: some additional data such that, when given the instance
and the certificate, it can indeed be verified in polynomial time that the
instance is a ``yes'' instance.  Obviously, P is included in~NP, and
probably the most outstanding open question in computer science is to
determine whether P is equal or different from NP; it is widely believed
that they are different.

A problem $A$ \emph{reduces} to a problem~$B$ (roughly) if, when given an
oracle to solve problem~$B$, one can solve problem~$A$ in polynomial time.
A problem~$A$ is \emph{NP-hard} if it is ``at least as hard'' as any
problem in NP, or more precisely if every problem in~NP reduces to it.  If
P and NP are different, then no NP-hard problem can be solved in polynomial
time.  Finally, a problem is \emph{NP-complete} if it is both NP-hard and
in NP.

\subsection{Boolean constraint satisfaction problems}

In this section, we recall a few basic results about Boolean constraint
satisfaction problems; our presentation is inspired from Dalmau and
Ford~\cite{df-gslov-03}.  For a more detailed account of this tremendous
body of research, we refer to the book of Creignou, Khanna, and
Sudan~\cite{cks-ccbcs-01}.

We start by introducing the \emph{generalized satisfiability problem}
$\SAT(R)$, which is a variant of the usual SAT problem. An $r$-ary
relation~$R$ is any nonempty subset of $\{0,1\}^r$. A $\CNF(R)$-formula is
a finite conjunction of clauses $C_1 \wedge \ldots \wedge C_n$ such that
each clause, $C_i$, is an \emph{atomic formula} of the form $R(v_1, \ldots
,v_r)$ where $v_1 \ldots v_r$ are Boolean variables. An atomic formula
$R(v_1, \ldots v_r)$ is \emph{satisfied} by a variable assignment $f : V
\rightarrow \{0,1\}$ if and only if $(f(v_1) \ldots f(v_r)) \in R$, and a
$\CNF(R)$ formula is \emph{satisfiable} if and only if there exists an
assignment satisfying all its clauses simultaneously. Each relation $R$
gives rise to the \emph{generalized satisfiability problem}~$\SAT(R)$:
given a $\CNF(R)$-formula, is it satisfiable?

It is sometimes convenient to assume that constants can appear in
$\CNF(R)$-formulas: Each clause is an atomic formula of the form
$R(v_1,\ldots,v_r)$ where each $v_i$ is a Boolean variable or a constant (0
or~1).  We call any formula obtained this way a
$\CNF_C(R)$-formula. Similarly, the problem
\emph{generalized satisfiability problem with constants},
$\SAT_C(R)$, is defined with $\CNF_C(R)$-formulas.

The computational complexity of the generalized satisfiability problem with
constants has been completely classified by Schaefer in a celebrated
paper~\cite{s-csp-78}.  We introduce the following definitions in order to
state this classification.  Here, the symbols $\wedge$, $\vee$,
and~$\oplus$: $\{0,1\}^r\times\{0,1\}^r\to\{0,1\}$ denote the usual logical
operations AND, OR and XOR, applied bit-wise.

A relation $R$ is 
\begin{itemize}
\item \emph{Horn} if $x,y \in R \rightarrow x \wedge y \in
R$, 
\item \emph{dual-Horn} if $x,y \in R \rightarrow x \vee y \in R$,
\item \emph{bijunctive} if $x,y,z \in R \rightarrow (x \wedge y) \vee (x \wedge
z) \vee (y \wedge z) \in R$, 
\item \emph{affine} if $x,y,z \in R \rightarrow x
\oplus y \oplus z \in R$.
\end{itemize}
A relation $R$ is \emph{Schaefer} if it is Horn, dual-Horn,
bijunctive, or affine.

\begin{theorem}[\cite{s-csp-78}]\label{T:schaefer}
  Let $R$ be a relation. If $R$ is Schaefer, then $\SAT_C(R)$ is in P,
  otherwise it is NP-hard.
\end{theorem}

For our reduction, we will restrict ourselves to constraint satisfaction
problems where the number of occurrences of every variable is at most 2.
We denote by $\SAT(2,R)$ the instances of $\SAT(R)$ in which every variable
occurs at most twice.  Similarly, we denote by $\SAT_C(2,R)$ the instances
of $\SAT_C(R)$ in which every variable occurs at most twice. The following
definition is key to the classification of these problems.

Let $R \subseteq \{0,1\}^r$ be a relation. Let $x,y,x' \in \{0,1\}^r$, then
$x'$ is a \emph{step} from $x$ to $y$ if $d(x,x')=1$ and
$d(x,x')+d(x',y)=d(x,y)$, where $d$ is the Hamming distance. $R$ is a
\emph{$\Delta$-matroid (relation)} if it satisfies the following
\emph{two-step axiom}:
  \begin{quote}
    For all $x,y \in R$ and for all $x'$ a step from $x$ to $y$, either
    $x'\in R$ or there exists $x'' \in R$ which is a step from $x'$ to $y$.
  \end{quote}

We now come to the classification theorem for $\SAT_C(2,R)$:
\begin{theorem}[Feder~\cite{f-flcs-01}]\label{T:feder}
  Let $R$ be a relation that is not a $\Delta$-matroid relation. Then
  $\SAT_C(2,R)$ is polynomially equivalent to $\SAT_C(R)$.
\end{theorem}

Theorems~\ref{T:schaefer} and~\ref{T:feder} immediately imply the
complexity result that we will use:
\begin{corollary}\label{C:complexity}
  Let $R$ be a relation that is not Schaefer (that is, not Horn, dual Horn,
  bijunctive, or affine) and not a $\Delta$-matroid.  Then $\SAT_C(2,R)$ is
  NP-hard.
\end{corollary}

%%%%%%%%%%%%%%%%%%%%%%%%%%%%%%%%%%%%%%%%%%%%%%%%%%%%%%%%%%%%%%%%%%%%%%%%%%
\section{NP-hardness of detecting immersibility}\label{S:thm}

In this section, we prove the following theorem.
\begin{theorem}\label{T:main}
  The problem \textsc{Immersibility} is NP-hard.
\end{theorem}

The proof of Theorem~\ref{T:main} will proceed by a reduction of
$\SAT_C(2,R)$ to the problem $\textsc{Immersibility}$ for a relation $R$
that is neither Schaefer nor a $\Delta$-matroid, which implies by
Corollary~\ref{C:complexity} that $\SAT_C(2,R)$ and hence
$\textsc{Immersibility}$ are NP-hard.

\subsection{Gadgets}

We now show how to reduce $\SAT_C(2,R)$ to the problem
$\textsc{Immersibility}$. To this end, we use a $6$-ary relation $R$ which
we will describe later. We start with a formula $\Phi$ that is, by
definition, a conjunction of clauses of the form $R(x_{i_1}, x_{i_2},
x_{i_3}, x_{i_4}, x_{i_5}, x_{i_6})$, where $x_i$ is either a variable or a
constant, and every variable appears at most twice in $\Phi$.

The gadgets that we use for our reduction are of three types.  Each clause
is represented by a \emph{clause gadget}, a block of six tetrahedra glued
together around an edge.  For each variable occurring exactly twice
in~$\Phi$, we connect these two occurrences in the clauses using
\emph{tubes}, which are also blocks of six tetrahedra.  Finally, the
\emph{constant gadgets} are used to represent the constants $0$ or $1$
appearing in the clauses. The idea for the proof is that a clause is
satisfiable if and only if the normal coordinates in the clause gadget are
immersible; the tubes then enforce consistency between the
clauses. Therefore, the whole formula will be satisfiable if and only if
the associated normal coordinates are immersible.

\paragraph*{The clause gadget.}

Consider the gadget $G$ pictured in Figure~\ref{F:clausegadget}. It
consists of six tetrahedra that all have an edge in common, and contain
each three normal disks: two triangles and one quadrilateral. For every
clause in $\Phi$, we create a copy of the clause gadget~$G$; these copies
will be connected using tubes, described below.

\begin{figure}[htb]
\centering
\def\svgwidth{5.5cm}
\executeiffilenewer{fig/maingadget.svg}{fig/maingadget.pdf}%
{inkscape -z -D --file=fig/maingadget.svg %
--export-pdf=fig/maingadget.pdf --export-latex}%
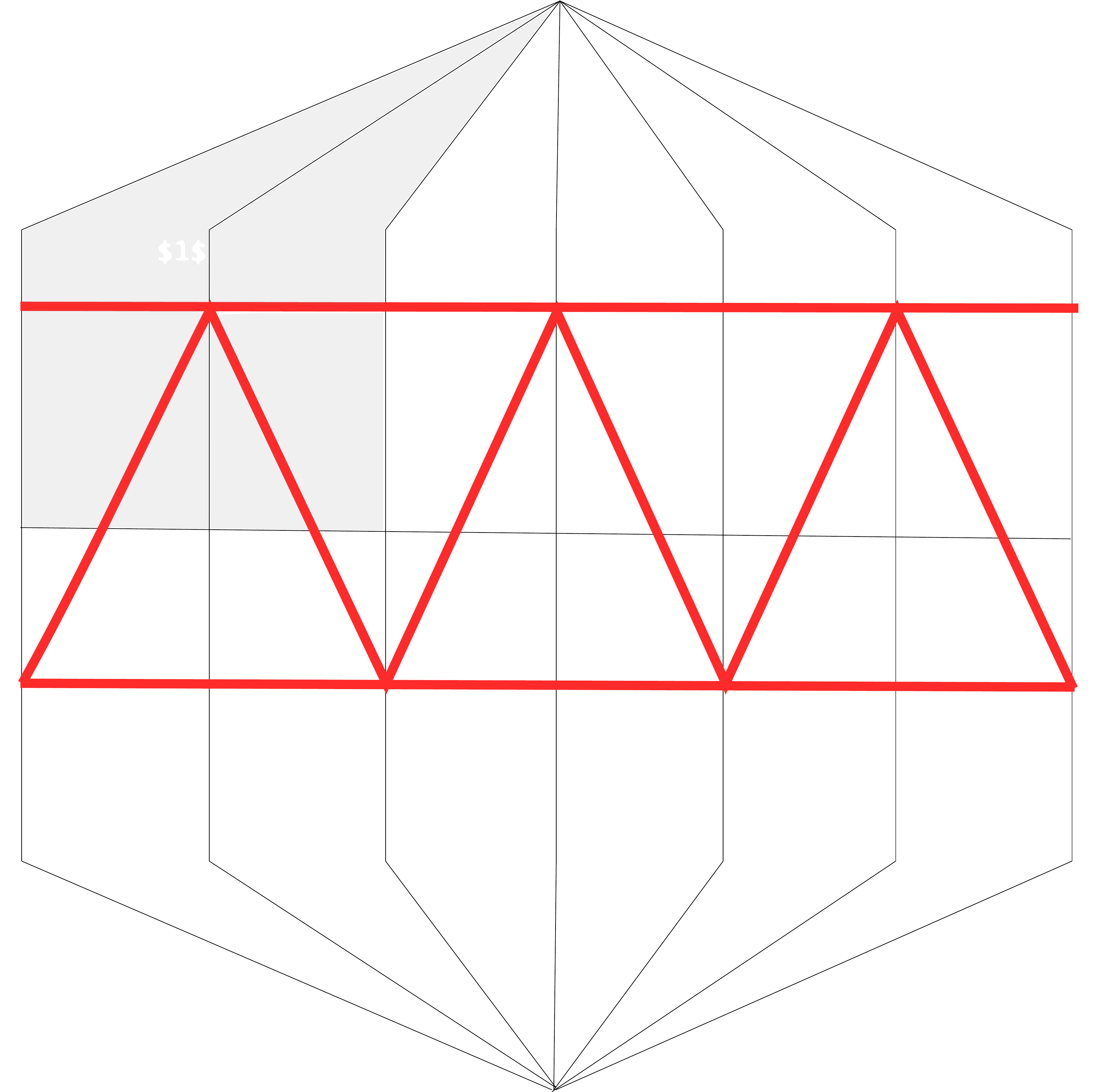%

\caption{The clause gadget.  The couple of faces corresponding to the
  variable~$x_1$ is shaded.}
\label{F:clausegadget}
\end{figure}

The rationale behind this gadget is the following. At the interface of two
adjacent tetrahedra, exactly two gluings can be done (see
Figure~\ref{F:gluings}). This choice of gluing can be described by a
variable $x_i \in \{0,1\}$, where $0$ corresponds to the gluing a. and $1$
to the gluing b.  Equivalently, a value of~$1$ corresponds to the fact that
the two block curves at the specified position cross. Therefore, each
variable in a clause has a pair of associated faces on the boundary of the
clause gadget; for example, the two shaded triangles in
Figure~\ref{F:clausegadget} are the pair of faces associated to
variable~$x_1$.

This way, a global gluing of the singular normal surface in $G$ is
described by an element $x \in \{0,1\}^6$. The order on the six variables
is pictured in Figure~\ref{F:clausegadget}.

\begin{figure}[htb]
\centering
\def\svgwidth{5.5cm}
\executeiffilenewer{fig/gluings.svg}{fig/gluings.pdf}%
{inkscape -z -D --file=fig/gluings.svg %
--export-pdf=fig/gluings.pdf --export-latex}%
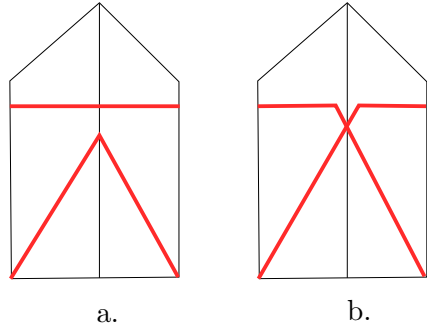%

\caption{a. Local gluing corresponding to $x_i=0$. b. Local gluing
  corresponding to $x_i=1$.}
\label{F:gluings}
\end{figure}

We define the following relation $R$ on $6$ variables:

\begin{eqnarray*}
R&=\ \ \big\{& (0,0,0,0,0,0);(0,0,0,1,0,1);(0,0,1,0,1,0);(0,1,0,0,0,1);(0,1,0,1,0,0);\\
 &
 &(0,1,1,0,1,1);(1,0,0,0,1,0);(1,0,1,0,0,0);(1,0,1,1,0,1);(1,1,0,1,1,0);\\
 &
 &(1,1,1,1,1,1)\ \ \big\}
\end{eqnarray*}

This allows us to get to the following lemma.

\begin{lemma}\label{L:relationGadget}
The singular normal surface in the gadget $G$ specified by the gluing $x
\in \{0,1\}^6$ is immersed if and only if $x \in R$.
\end{lemma}

\begin{proof}[Proof of Lemma~\ref{L:relationGadget}]
  The proof is done by exhaustive checking, i.e., checking for every
  possible 6-tuple whether there is a branch point around the central edge
  or not. As an example, Figure~\ref{F:relation} pictures the singular
  normal surfaces obtained with the global gluings $(1,0,1,1,0,1)$ and
  $(1,0,1,1,1,1)$, yielding in one case an immersed normal surface and in
  the other a singular normal surface with a branch point.
\end{proof}

\begin{figure}[htb]
\centering
\def\svgwidth{8cm}
\executeiffilenewer{fig/relation.svg}{fig/relation.pdf}%
{inkscape -z -D --file=fig/relation.svg %
--export-pdf=fig/relation.pdf --export-latex}%
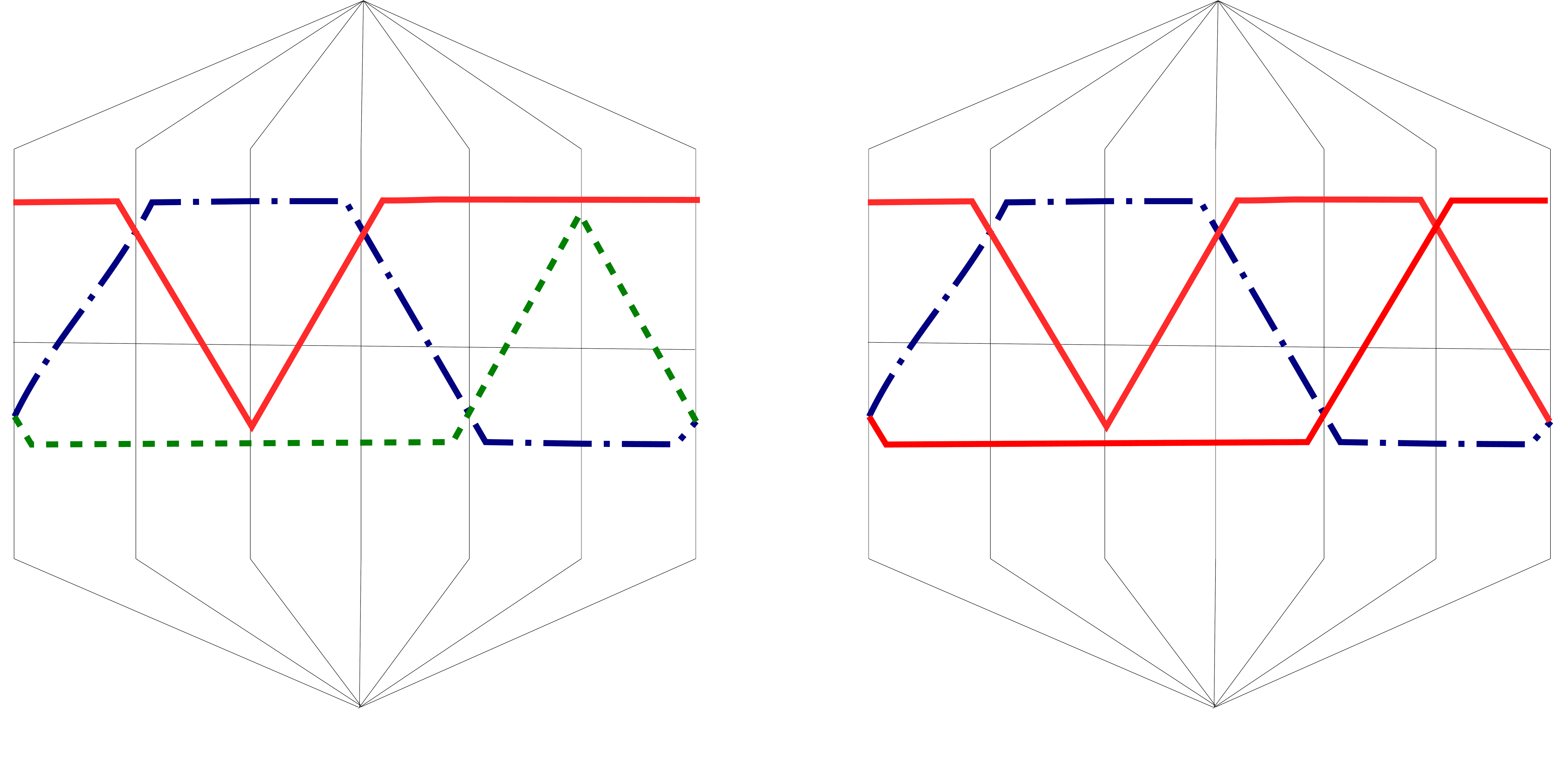%

\caption{(a) Immersed normal surface corresponding to coordinates
  (1,0,1,1,0,1), yielding three block curves.  (b) Singular normal surface
  corresponding to coordinates (1,0,1,1,1,1), yielding two block
  curves. The one winding twice around the central edge witnesses a branch
  point.}
\label{F:relation}
\end{figure}

Moreover:
\begin{proposition}\label{P:propSchaefer}
  $SAT_C(2,R)$ is NP-hard.
\end{proposition}
\begin{proof}
  By Corollary~\ref{C:complexity}, it suffices to prove that the relation
  $R$ is not Schaefer, that is, neither (i) Horn, (ii) dual Horn, (iii)
  bijunctive, nor (iv) affine; and that, furthermore, it is (v) not a
  $\Delta$-matroid.  The proofs are somewhat tedious but straightforward.
  \begin{enumerate}
  \item[(i)] $(1,0,1,0,0,0)$ and $(1,1,0,1,1,0)$ are in $R$, but their
    conjunction $(1,0,0,0,0,0)$ is not.
  \item[(ii)] $(1,0,1,0,0,0)$ and $(1,1,0,1,1,0)$ are in $R$, but their
    disjunction $(1,1,1,1,1,0)$ is not.
  \item[(iii)] If we take $x=(1,0,1,0,0,0)$, $y=(1,1,0,1,1,0)$, and
    $z=(0,0,0,0,0,0)$, $(x \wedge y) \vee (x \wedge z) \vee (y \wedge z) =
    (1,0,0,0,0,0)$, which is not in $R$.
  \item[(iv)] If we take $x=(1,0,1,0,0,0), y=(1,1,0,1,1,0)$, and
    $z=(0,0,0,0,0,0)$, we have $x \oplus y \oplus z=(0,1,1,1,1,0)$, which
    is not in $R$.
  \item[(v)] We take $x=(1,1,1,1,1,1)$, $y=(1,0,0,0,1,0)$, and
    $x'=(1,0,1,1,1,1)$, $x$ and $y$ are in $R$ and $x'$ is a step from $x$
    to $y$, but $x'$ is not in $R$ and there does not exist any $x'' \in R$
    which is a step from $x'$ to $y$.\qedhere
  \end{enumerate}
\end{proof}

\paragraph*{The tubes.}

A \emph{tube} is the block pictured in Figure~\ref{F:tubes}; it is
comprised of six tetrahedra that all have one edge in common, and contain
each two normal disks: either a pair of one triangle and one quadrilateral,
or a pair of triangles of different types. As in Figure~\ref{F:tubes}, we
denote by $A_1$ and $A_2$ the two pairs of faces that are crossed by two
adjacent quadrilaterals.

Similarly as for the clause gadget, the gluings in a tube gadget at $A_1$
and $A_2$ can be specified by two variables, again with the convention of
Figure~\ref{F:gluings}: $0$ and $1$ respectively for non-crossing and
crossing block curves.

\begin{lemma}\label{L:relationTube}
  The singular normal surface in a tube specified by the gluing is immersed
  if and only if both variables of the tube are equal.
\end{lemma}
\begin{proof}
  This is immediate by checking the four possible assignments of the
  variables.
\end{proof}

\begin{figure}[htb]
\centering
\def\svgwidth{5.5cm}
\executeiffilenewer{fig/tubes.svg}{fig/tubes.pdf}%
{inkscape -z -D --file=fig/tubes.svg %
--export-pdf=fig/tubes.pdf --export-latex}%
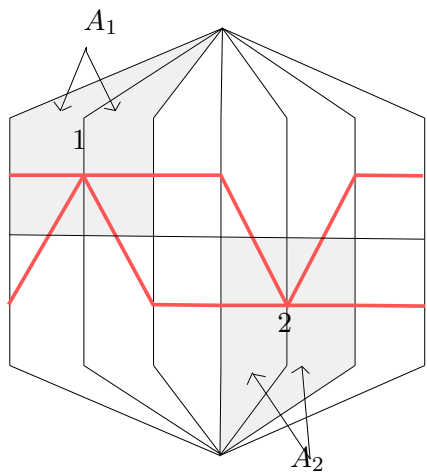%

\caption{A tube, where the shaded regions denote the locations where the
  tube will be connected to clauses.}
\label{F:tubes}
\end{figure}

Consider a variable~$v$ appearing exactly twice in~$\Phi$, in clauses $C_1$
and~$C_2$ (it may be that $C_1=C_2$).  Let $G_1$ and~$G_2$ be the copies of
the clause gadget corresponding to $C_1$ and~$C_2$, respectively; let $B_1$
be the pair of faces in~$G_1$ corresponding to the occurrence of the
variable~$v$ in~$G_1$, and similarly let $B_2$ be the other pair of faces
in~$G_2$ corresponding to the occurrence of the variable~$v$ in~$G_2$.  We
create a tube~$T_v$, and we glue the pair of faces~$A_1$ to~$B_1$
and~$A_2$ to~$B_2$.

\paragraph*{The constants.}

A constant gadget is one of the blocks pictured in Figure~\ref{F:constants}
(since each gadget only consists of one tetrahedron, we did not adopt the
schematic representation in this drawing). The gadget $CG_0$ for the
constant $0$ consists of one tetrahedron containing two triangles, while
the gadget $CG_1$ for the constant $1$ consists of one tetrahedron
containing two crossing quadrilaterals.  In both cases, the pair of
faces~$A$, at which the constant will be connected to a clause gadget, is
made of the two front faces in Figure~\ref{F:constants}.

\begin{figure}[htb]
\centering
\def\svgwidth{8cm}
\executeiffilenewer{fig/constants.svg}{fig/constants.pdf}%
{inkscape -z -D --file=fig/constants.svg %
--export-pdf=fig/constants.pdf --export-latex}%
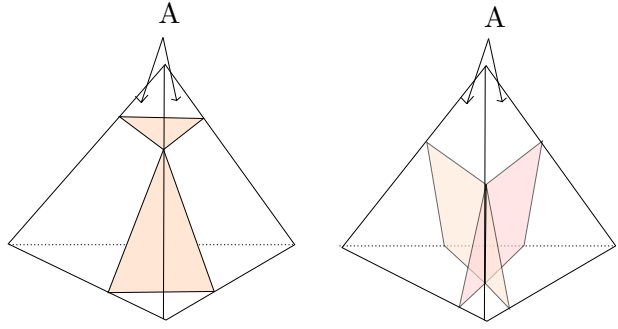%

\caption{Left: The gadget~$CG_0$.  Right: The gadget~$CG_1$.  In both
  cases, the two front faces constitute the pair~$A$.  The trace of the
  singular normal surface on~$A$ is the same in both cases.}
\label{F:constants}
\end{figure}

Whenever a constant~$0$ appears in~$\Phi$, we create a copy of the constant
gadget~$CG_0$, and attach the corresponding pair of faces in the clause
gadget to the pair~$A$ in the constant gadget.  The same holds for the
constant~$1$, with $CG_1$ instead of~$CG_0$.

\subsection{Proof of the reduction}

We now have all the tools to prove Theorem~\ref{T:main}. Starting with a
formula $\Phi$, we build a triangulation $T$ and normal coordinates $N$
with the clause gadgets, the tubes, and the constant gadgets. We first prove
that this triangulation forms a 3-manifold with boundary.

\begin{proposition}\label{P:3manifold}
The triangulation $T$ corresponding to a formula $\Phi$ is a 3-manifold
with boundary.
\end{proposition}

\begin{proof}[Proof of Proposition~\ref{P:3manifold}]
  First, we show that every vertex $v$ of~$T$ has a neighborhood
  homeomorphic to the closed half-space. The vertex $v$ is adjacent to a
  clause gadget $C$ as well as between zero and three tubes or constant
  gadgets. We focus on the case with three tubes $T_1, T_2$ and $T_3$, the other ones being
  handled similarly. Since gadgets are glued along discs with disjoint
  interiors, and the tubes are locally disjoint except at their
  intersection with $C$, the neighborhood around $v$ is the one pictured in
  Figure~\ref{F:manifold}. One readily sees that this neighborhood is
  homeomorphic to a half-space.

\begin{figure}[htb]
\centering
\def\svgwidth{5cm}
\executeiffilenewer{fig/manifold.svg}{fig/manifold.pdf}%
{inkscape -z -D --file=fig/manifold.svg %
--export-pdf=fig/manifold.pdf --export-latex}%
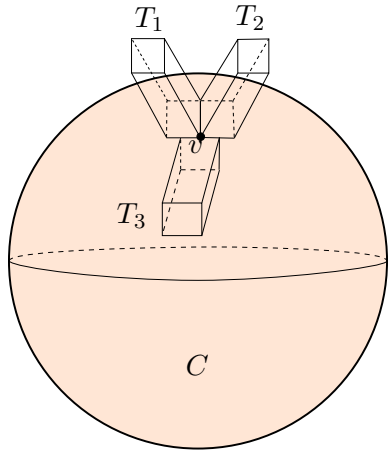%

\caption{The local picture around a vertex $v$ of the triangulation $T$.}
\label{F:manifold}
\end{figure}

  Moreover, by construction, no edge is identified to itself in reverse.
  Lemma~\ref{L:manifold} concludes the proof.
\end{proof}

\begin{proposition}\label{P:mainprop}
  The normal coordinates $N$ correspond to an immersible surface if and
  only if the formula $\Phi$ is satisfiable.
\end{proposition}
\begin{proof}[Proof of Proposition~\ref{P:mainprop}]
  We will use the following observation.  Consider an edge~$e$ between the
  two faces forming a pair of faces in a gadget.  The block around~$e$ is
  made of three or four tetrahedra, depending on whether a constant gadget
  or a tube is glued to the pair on the clause gadget.  Then the singular
  surface has no branch point on~$e$ if and only if the values of the
  variables on the two pair of faces are identical, and, in case one of the
  gadget is a constant gadget, the variable takes the value specified by
  that gadget.

  Assume first that the normal coordinates $N$ correspond to an immersible
  surface. The gluings corresponding to variables appearing exactly once in
  $\Phi$ define a partial assignment of the variables. For the ones
  appearing twice, Lemma~\ref{L:relationTube} implies that the two
  variables in the same tube are equal.  Since they are glued to the pairs
  of faces representing their two occurrences in the clause gadgets, and by
  the observation, this defines a consistent assignment of the variables;
  moreover, again by the observation, the variables on the pair of faces on
  the clause gadgets corresponding to the constants take the appropriate
  values.  Finally, since no branch point arises in a clause gadget,
  Lemma~\ref{L:relationGadget} implies that each clause is satisfied by the
  assignment of variables.

  Conversely, assume that $\Phi$ is satisfiable.  The satisfying assignment
  naturally defines the values of the variables on each pair of faces.  We
  prove that these gluing rules do not create branch points. Branch points
  only occur on non-boundary edges of the triangulation; each such edge is
  either an edge in a clause gadget, an edge in a tube gadget, or an edge
  at the interface between two pairs of faces.  No such edge contains
  contains a branch point, by Lemmas \ref{L:relationGadget}
  and~\ref{L:relationTube} and the observation, respectively.  This
  concludes the proof.
\end{proof}

\section{Variants}\label{S:variants}

We now establish a few corollaries, settling the complexity of a few
variants of the \textsc{Immersibility} problem.

The first one settles the complexity of the problem
$k$-\textsc{Bounded-Immersibility}, which is the one of testing immersibility
when all the normal coordinates are bounded by $k \geq 1$.

\begin{corollary}\label{C:NPcomplete}
  The problem $k$-\textsc{Bounded-Immersibility} is NP-complete.
\end{corollary}
(Recall that a problem is \emph{NP-complete} if it is both in NP and
NP-hard.)
\begin{proof}[Proof of Corollary~\ref{C:NPcomplete}]
  The NP-hardness reduction for Theorem~\ref{T:main} only involves normal
  coordinates bounded by $1$, so $k$-\textsc{Bounded-Immersibility} is
  NP-hard as well.  

  The certificate we use to prove the membership in NP is the global gluing
  for an immersed surface. At an interface, the local gluing can be
  described by a permutation in $S_k$, which is of finite size since $k$ is
  fixed. Since the number of interfaces is bounded by the size of the
  triangulation, the certificate has polynomial size. As mentioned in the
  preliminaries, when one is provided with normal coordinates and a global
  gluing, one can test in linear time whether the corresponding singular
  normal surface is immersed.
\end{proof}

The gadget we use to show the NP-hardness of \textsc{Immersibility} only
outputs a 3-manifold with boundary, but can be easily tweaked to handle the
problem with boundaryless manifolds, which we name
\textsc{Boundaryless-Immersibility}.

\begin{corollary}\label{C:boundaryless}
The problem \textsc{Boundaryless-Immersibility} is NP-hard.
\end{corollary}

\begin{proof}[Proof of Corollary~\ref{C:boundaryless}]
  The triangulation $T_{\Phi}$ with normal coordinates $N$ obtained by
  the reduction in Theorem~\ref{T:main} can become boundaryless by
  \emph{doubling} it, i.e., by taking $T_{\Phi} \bigsqcup T_{\Phi}$ and
  gluing one onto the other with the identity homeomorphism on their
  boundaries. It is straightforward to check that the resulting space is a
  $3$-manifold, and that $N$ is immersible in
  $T_{\Phi}$ if and only if $N \bigsqcup N$ is immersible in $T_{\Phi} \bigsqcup
  T_{\Phi}$.
\end{proof}

As we mentioned in the preliminaries, the triangulations we consider in
this chapter may not be simplicial complexes, and indeed the gadgets we use
in our reduction may display some slightly pathological behavior. For
example, if a tube links the variables $1$ and $5$ of a clause gadget, the
central edge of the tube has its endpoints identified. This can be avoided
by replacing each tube gadget by two copies of itself, glued to each other
along the pairs $A_1$ and $A_2$. These ``double tubes'' are glued to the
clause gadgets as usual ones. Then it can be checked that the resulting
triangulation is a simplicial complex, and this shows that
\textsc{Immersibility} is also NP-hard when the triangulation is a
simplicial complex.

As a last remark, we note that since the triangulation obtained by the
reduction in Theorem~\ref{T:main} consists of balls linked by tubes, it can
be embedded in $\R^3$. This shows that \textsc{Immersibility} is NP-hard
even when restricted to submanifolds of $\R^3$, which is for example the
case of knot complements.

\section{Testing local immersibility}\label{S:local}

In this section, we provide a polynomial-time algorithm to test whether
normal coordinates are immersible when the triangulation is just a block,
i.e., a collection of tetrahedra glued around a single edge. When applied
to every edge of a more complicated triangulation, this provides a
\emph{local} test for immersibility, in that it detects local
obstructions but not more involved ones, as pictured in
Figure~\ref{F:badlocal}.

\begin{figure}
\centering
\def\svgwidth{12cm}
\executeiffilenewer{fig/badlocal.svg}{fig/badlocal.pdf}%
{inkscape -z -D --file=fig/badlocal.svg %
--export-pdf=fig/badlocal.pdf --export-latex}%
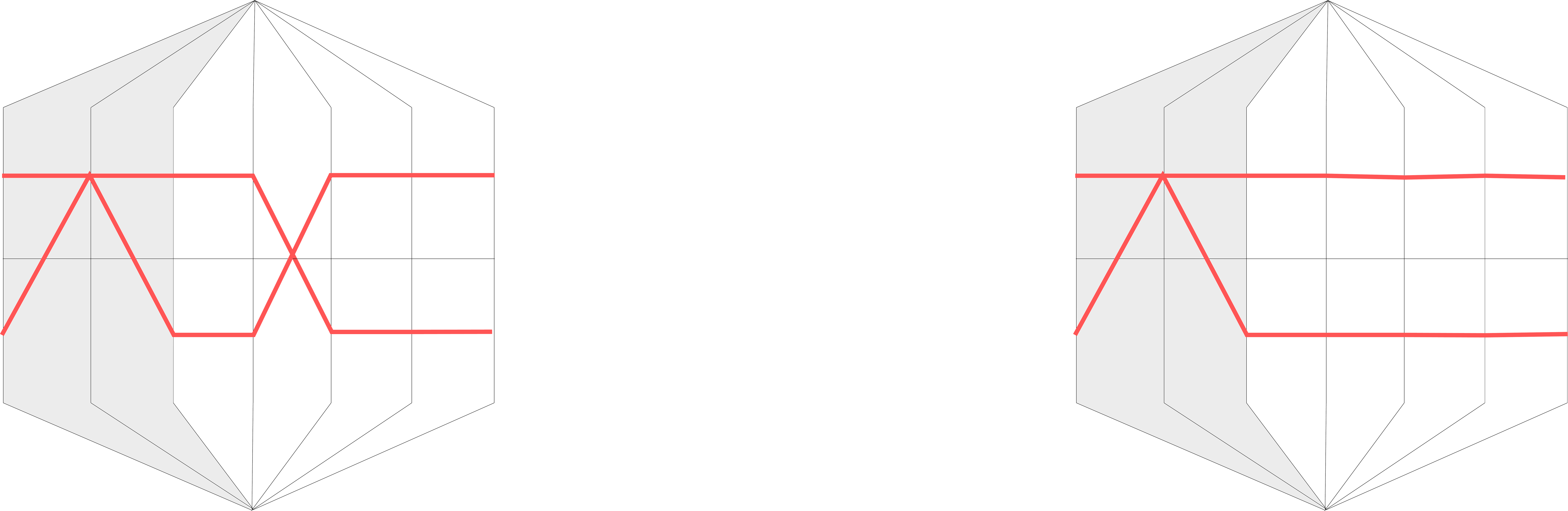%

\caption{When these two blocks are made into a single triangulation by gluing
them along the shaded regions, the normal coordinates are locally
immersible around every edge, but not globally immersible.}
\label{F:badlocal}
\end{figure}

For the rest of this section, we denote by $T$ a triangulation comprised of
a single block, and by $N$ some normal coordinates for this triangulation.
The algorithm works by reducing the problem to a maximum flow computation
in a directed graph with integer capacities.  In this context, given a
directed graph~$G$ where each edge~$e$ bears a nonnegative integer
\emph{capacity}~$c(e)$, and two vertices $s$ and~$t$, a flow can be seen as
a set of paths in~$G$ from $s$ to~$t$ such that each edge~$e$ is used at
most $c(e)$ times by the paths; a maximum flow is a flow with a maximum
number of paths.  It is well-known that maximum flows in directed graphs
can be computed in small polynomial time~\cite[Ch.~10]{s-cope-03}.

In our problem, the construction of the directed graph~$G$ corresponding to
a block and to given normal coordinates is pictured in
Figure~\ref{F:blocktograph}: Every triangle type or quadrilateral type is
associated to a directed edge, and the corresponding normal coordinate is
carried over as the capacity of that edge.  We denote by $s_1$ and $s_2$
the extremal vertices of this graph on one side of the block, and $t_1$ and
$t_2$ the vertices on the other side. Note that in the triangulation, both
sides of the block are identified, and therefore for $i=1,2$, $s_i$
corresponds to the same vertex as $t_i$, which we denote by $x_1$
and~$x_2$. The algorithm then simply computes the maximum flow between
$s_1$ and $t_1$, and checks whether it is equal to the sum of the
capacities of the edges going out of $s_1$.

\begin{figure}
\centering
\def\svgwidth{12cm}
\executeiffilenewer{fig/blocktograph.svg}{fig/blocktograph.pdf}%
{inkscape -z -D --file=fig/blocktograph.svg %
--export-pdf=fig/blocktograph.pdf --export-latex}%
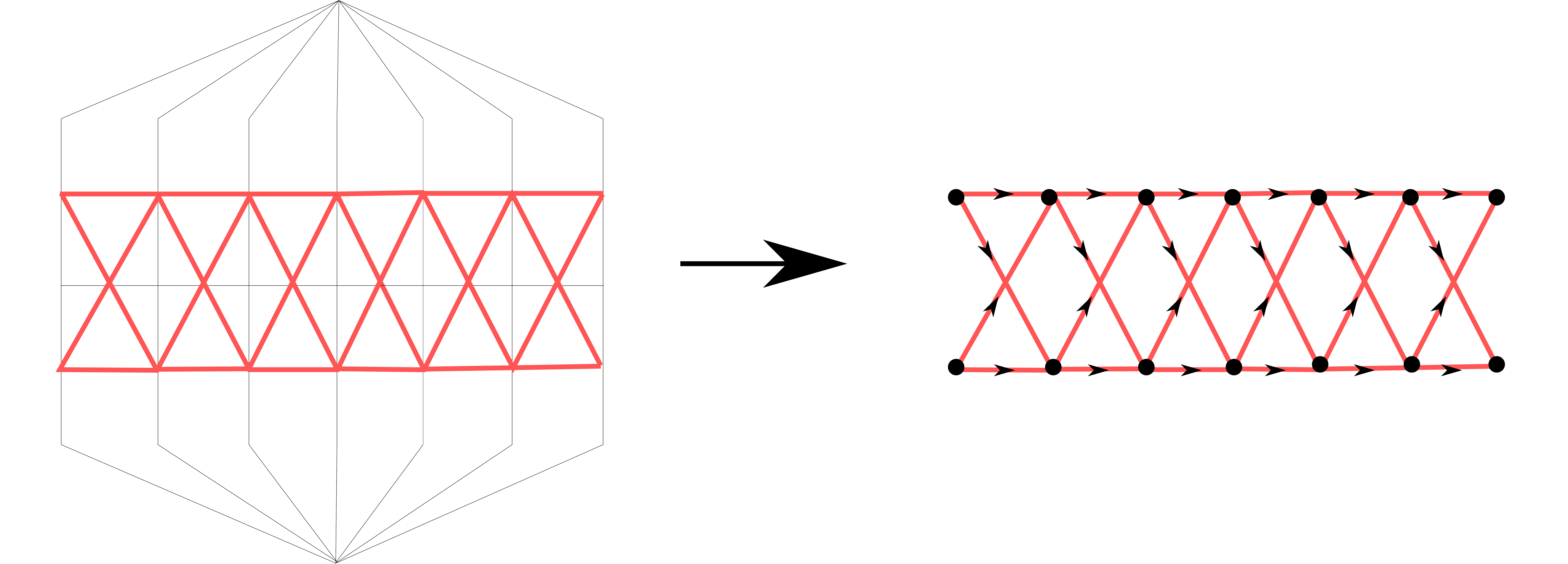%

\caption{The directed weighted graph $G$ is obtained by choosing an
  orientation around the block, adding a single directed edge for every
  normal disk type, and setting the capacity of an edge to be equal to the
  corresponding normal coordinate.}
\label{F:blocktograph}
\end{figure}

The correctness of this algorithm is warranted by the following lemma.

\begin{lemma}\label{L:local-immersibility}
  The normal coordinates $N$ are immersible if and only if the maximum flow
  between $s_1$ and $t_1$ equals the sum of the capacities of the edges
  leaving $s_1$.
\end{lemma}

\begin{proof}[Proof of Lemma~\ref{L:local-immersibility}]
  For the first implication, if the normal coordinates are immersible,
  there exists a gluing inducing no branch point. The associated normal
  surface is therefore a union of disks winding exactly once around the
  central edge. Every such disk corresponds in $G$ to a path going from
  $s_1$ to~$t_1$ or from $s_2$ to~$t_2$. Therefore, the maximum flow going
  from $s_1$ to $t_1$ is the sum of the capacities going out of $s_1$.

  In the other direction, such a maximum flow naturally defines a gluing
  corresponding to an immersed surface. Indeed, the flow can be seen as a
  set of paths in~$G$ from $s_1$ to~$t_1$, which in turn correspond to
  cycles when identifying back $s_1$ and~$t_1$.  This gives a partial
  gluing of the normal disks, which yields a (not necessarily disjoint)
  union of embedded disks passing through~$x_1$.  It is also easy to see
  that, if we decrease the capacities of each edge of~$G$ by the number of
  times it is used by the flow, then, in this resulting \emph{residual
    graph}, there exists a maximum flow from $s_2$ to~$t_2$ that uses all
  residual capacities.  (This could also be checked using the same
  algorithm.)  This flow similarly corresponds to normal disks.  All in
  all, we obtain an immersed surface whose normal coordinates are precisely
  the capacities of the original graph, which completes the proof.
\end{proof}

It is an intriguing problem to extend this algorithm to deal with more
complicated triangulations. With some work, we can for example obtain a
polynomial-time algorithm to test immersibility for the triangulation
pictured in Figure~\ref{F:badlocal}. Finding a general polynomial-time
algorithm is ruled out by our main Theorem~\ref{T:main} though.  To
conclude, it would be interesting to determine whether the immersibility
problem is \emph{fixed-parameter tractable}~\cite{f-pcndr-01} in the size
of the triangulation~$T$; or, in other words, whether there exists an
algorithm with running time $f(|T|).\text{poly}(n)$ for this purpose, where
$f(|T|)$ is an arbitrary function depending on the number of simplices
of~$T$ and $\text{poly}(n)$ is an arbitrary polynomial.

\bibliographystyle{siam}
\bibliography{bibexport,arnaud}
%%%%%%%%%%%%%%%%%%%%%%%%%%%%%%%%%%%%%%%%%%%%%%%%%%%%%%%%%%%%%%%%%%%%%%%%%%
%%%%%%%%%%%%%%%%%%%%%%%%%%%%%%%%%%%%%%%%%%%%%%%%%%%%%%%%%%%%%%%%%%%%%%%%%%
%%%%%%%%%%%%%%%%%%%%%%%%%%%%%%%%%%%%%%%%%%%%%%%%%%%%%%%%%%%%%%%%%%%%%%%%%%
\end{document}

%% file: Normaldisks.pdf_tex
%% Creator: Inkscape inkscape 0.48.4, www.inkscape.org
%% Accompanies image file 'Normaldisks.pdf' (pdf, eps, ps)
\begingroup%
  \makeatletter%
  \providecommand\color[2][]{%
    \errmessage{(Inkscape) Color is used for the text in Inkscape, but the package 'color.sty' is not loaded}%
    \renewcommand\color[2][]{}%
  }%
  \providecommand\transparent[1]{%
    \errmessage{(Inkscape) Transparency is used (non-zero) for the text in Inkscape, but the package 'transparent.sty' is not loaded}%
    \renewcommand\transparent[1]{}%
  }%
  \providecommand\rotatebox[2]{#2}%
  \ifx\svgwidth\undefined%
    \setlength{\unitlength}{5454.25bp}%
    \ifx\svgscale\undefined%
      \relax%
    \else%
      \setlength{\unitlength}{\unitlength * \real{\svgscale}}%
    \fi%
  \else%
    \setlength{\unitlength}{\svgwidth}%
  \fi%
  \global\let\svgwidth\undefined%
  \global\let\svgscale\undefined%
  \makeatother%
  \begin{picture}(1,0.23057306)%
    \put(0,0){\includegraphics[width=\unitlength]{Normaldisks.pdf}}%
  \end{picture}%
\endgroup%

%% file: singular.pdf_tex
%% Creator: Inkscape inkscape 0.48.3.1, www.inkscape.org
%% Accompanies image file 'singular.pdf' (pdf, eps, ps)
\begingroup%
  \makeatletter%
  \providecommand\color[2][]{%
    \errmessage{(Inkscape) Color is used for the text in Inkscape, but the package 'color.sty' is not loaded}%
    \renewcommand\color[2][]{}%
  }%
  \providecommand\transparent[1]{%
    \errmessage{(Inkscape) Transparency is used (non-zero) for the text in Inkscape, but the package 'transparent.sty' is not loaded}%
    \renewcommand\transparent[1]{}%
  }%
  \providecommand\rotatebox[2]{#2}%
  \ifx\svgwidth\undefined%
    \setlength{\unitlength}{418.5bp}%
    \ifx\svgscale\undefined%
      \relax%
    \else%
      \setlength{\unitlength}{\unitlength * \real{\svgscale}}%
    \fi%
  \else%
    \setlength{\unitlength}{\svgwidth}%
  \fi%
  \global\let\svgwidth\undefined%
  \global\let\svgscale\undefined%
  \makeatother%
  \begin{picture}(1,1.50191159)%
    \put(0,0){\includegraphics[width=\unitlength]{singular.pdf}}%
  \end{picture}%
\endgroup%

%% file: abstraction.pdf_tex
%% Creator: Inkscape inkscape 0.48.3.1, www.inkscape.org
%% Accompanies image file 'abstraction.pdf' (pdf, eps, ps)
\begingroup%
  \makeatletter%
  \providecommand\color[2][]{%
    \errmessage{(Inkscape) Color is used for the text in Inkscape, but the package 'color.sty' is not loaded}%
    \renewcommand\color[2][]{}%
  }%
  \providecommand\transparent[1]{%
    \errmessage{(Inkscape) Transparency is used (non-zero) for the text in Inkscape, but the package 'transparent.sty' is not loaded}%
    \renewcommand\transparent[1]{}%
  }%
  \providecommand\rotatebox[2]{#2}%
  \ifx\svgwidth\undefined%
    \setlength{\unitlength}{3054.825bp}%
    \ifx\svgscale\undefined%
      \relax%
    \else%
      \setlength{\unitlength}{\unitlength * \real{\svgscale}}%
    \fi%
  \else%
    \setlength{\unitlength}{\svgwidth}%
  \fi%
  \global\let\svgwidth\undefined%
  \global\let\svgscale\undefined%
  \makeatother%
  \begin{picture}(1,0.91860707)%
    \put(0,0){\includegraphics[width=\unitlength]{abstraction.pdf}}%
    \put(0.15786252,0.48342172){\color[rgb]{0,0,0}\makebox(0,0)[lb]{\smash{(a)}}}%
    \put(0.52011637,0.4834166){\color[rgb]{0,0,0}\makebox(0,0)[lb]{\smash{(b)}}}%
    \put(0.85575093,0.48342172){\color[rgb]{0,0,0}\makebox(0,0)[lb]{\smash{(c)}}}%
    \put(0.15323308,0.00138101){\color[rgb]{0,0,0}\makebox(0,0)[lb]{\smash{(d)}}}%
    \put(0.65321283,0.00138613){\color[rgb]{0,0,0}\makebox(0,0)[lb]{\smash{(e)}}}%
  \end{picture}%
\endgroup%

%% file: usual.pdf_tex
%% Creator: Inkscape inkscape 0.48.3.1, www.inkscape.org
%% Accompanies image file 'usual.pdf' (pdf, eps, ps)
\begingroup%
  \makeatletter%
  \providecommand\color[2][]{%
    \errmessage{(Inkscape) Color is used for the text in Inkscape, but the package 'color.sty' is not loaded}%
    \renewcommand\color[2][]{}%
  }%
  \providecommand\transparent[1]{%
    \errmessage{(Inkscape) Transparency is used (non-zero) for the text in Inkscape, but the package 'transparent.sty' is not loaded}%
    \renewcommand\transparent[1]{}%
  }%
  \providecommand\rotatebox[2]{#2}%
  \ifx\svgwidth\undefined%
    \setlength{\unitlength}{5056.43568bp}%
    \ifx\svgscale\undefined%
      \relax%
    \else%
      \setlength{\unitlength}{\unitlength * \real{\svgscale}}%
    \fi%
  \else%
    \setlength{\unitlength}{\svgwidth}%
  \fi%
  \global\let\svgwidth\undefined%
  \global\let\svgscale\undefined%
  \makeatother%
  \begin{picture}(1,0.29768852)%
    \put(0,0){\includegraphics[width=\unitlength]{usual.pdf}}%
    \put(0.12446628,0.00083742){\color[rgb]{0,0,0}\makebox(0,0)[lb]{\smash{(a)}}}%
    \put(0.50109368,0.00083433){\color[rgb]{0,0,0}\makebox(0,0)[lb]{\smash{(b)}}}%
    \put(0.87326485,0.00083742){\color[rgb]{0,0,0}\makebox(0,0)[lb]{\smash{(c)}}}%
  \end{picture}%
\endgroup%

%% file: branch.pdf_tex
%% Creator: Inkscape inkscape 0.48.3.1, www.inkscape.org
%% Accompanies image file 'branch.pdf' (pdf, eps, ps)
\begingroup%
  \makeatletter%
  \providecommand\color[2][]{%
    \errmessage{(Inkscape) Color is used for the text in Inkscape, but the package 'color.sty' is not loaded}%
    \renewcommand\color[2][]{}%
  }%
  \providecommand\transparent[1]{%
    \errmessage{(Inkscape) Transparency is used (non-zero) for the text in Inkscape, but the package 'transparent.sty' is not loaded}%
    \renewcommand\transparent[1]{}%
  }%
  \providecommand\rotatebox[2]{#2}%
  \ifx\svgwidth\undefined%
    \setlength{\unitlength}{3270.425bp}%
    \ifx\svgscale\undefined%
      \relax%
    \else%
      \setlength{\unitlength}{\unitlength * \real{\svgscale}}%
    \fi%
  \else%
    \setlength{\unitlength}{\svgwidth}%
  \fi%
  \global\let\svgwidth\undefined%
  \global\let\svgscale\undefined%
  \makeatother%
  \begin{picture}(1,0.42475366)%
    \put(0,0){\includegraphics[width=\unitlength]{branch.pdf}}%
  \end{picture}%
\endgroup%

%% file: maingadget.pdf_tex
%% Creator: Inkscape inkscape 0.48.3.1, www.inkscape.org
%% Accompanies image file 'maingadget.pdf' (pdf, eps, ps)
\begingroup%
  \makeatletter%
  \providecommand\color[2][]{%
    \errmessage{(Inkscape) Color is used for the text in Inkscape, but the package 'color.sty' is not loaded}%
    \renewcommand\color[2][]{}%
  }%
  \providecommand\transparent[1]{%
    \errmessage{(Inkscape) Transparency is used (non-zero) for the text in Inkscape, but the package 'transparent.sty' is not loaded}%
    \renewcommand\transparent[1]{}%
  }%
  \providecommand\rotatebox[2]{#2}%
  \ifx\svgwidth\undefined%
    \setlength{\unitlength}{1339.09469604bp}%
    \ifx\svgscale\undefined%
      \relax%
    \else%
      \setlength{\unitlength}{\unitlength * \real{\svgscale}}%
    \fi%
  \else%
    \setlength{\unitlength}{\svgwidth}%
  \fi%
  \global\let\svgwidth\undefined%
  \global\let\svgscale\undefined%
  \makeatother%
  \begin{picture}(1,0.99380574)%
    \put(0,0){\includegraphics[width=\unitlength]{maingadget.pdf}}%
    \put(0.14285251,0.75660878){\color[rgb]{0,0,0}\makebox(0,0)[lb]{\smash{$1$}}}%
    \put(0.31746055,0.30640986){\color[rgb]{0,0,0}\makebox(0,0)[lb]{\smash{$2$}}}%
    \put(0.47356167,0.75660878){\color[rgb]{0,0,0}\makebox(0,0)[lb]{\smash{$3$}}}%
    \put(0.63690466,0.30640986){\color[rgb]{0,0,0}\makebox(0,0)[lb]{\smash{$4$}}}%
    \put(0.80427082,0.75660878){\color[rgb]{0,0,0}\makebox(0,0)[lb]{\smash{$5$}}}%
    \put(0.95634877,0.30640986){\color[rgb]{0,0,0}\makebox(0,0)[lb]{\smash{$6$}}}%
    \put(-0.00198362,0.30640986){\color[rgb]{0,0,0}\makebox(0,0)[lb]{\smash{$6$}}}%
  \end{picture}%
\endgroup%

%% file: gluings.pdf_tex
%% Creator: Inkscape inkscape 0.48.3.1, www.inkscape.org
%% Accompanies image file 'gluings.pdf' (pdf, eps, ps)
\begingroup%
  \makeatletter%
  \providecommand\color[2][]{%
    \errmessage{(Inkscape) Color is used for the text in Inkscape, but the package 'color.sty' is not loaded}%
    \renewcommand\color[2][]{}%
  }%
  \providecommand\transparent[1]{%
    \errmessage{(Inkscape) Transparency is used (non-zero) for the text in Inkscape, but the package 'transparent.sty' is not loaded}%
    \renewcommand\transparent[1]{}%
  }%
  \providecommand\rotatebox[2]{#2}%
  \ifx\svgwidth\undefined%
    \setlength{\unitlength}{1237.225bp}%
    \ifx\svgscale\undefined%
      \relax%
    \else%
      \setlength{\unitlength}{\unitlength * \real{\svgscale}}%
    \fi%
  \else%
    \setlength{\unitlength}{\svgwidth}%
  \fi%
  \global\let\svgwidth\undefined%
  \global\let\svgscale\undefined%
  \makeatother%
  \begin{picture}(1,0.7712046)%
    \put(0,0){\includegraphics[width=\unitlength]{gluings.pdf}}%
    \put(0.20896707,0.00036624){\color[rgb]{0,0,0}\makebox(0,0)[lb]{\smash{a.}}}%
    \put(0.80938911,0.00590863){\color[rgb]{0,0,0}\makebox(0,0)[lb]{\smash{b.}}}%
  \end{picture}%
\endgroup%

%% file: relation.pdf_tex
%% Creator: Inkscape inkscape 0.48.3.1, www.inkscape.org
%% Accompanies image file 'relation.pdf' (pdf, eps, ps)
\begingroup%
  \makeatletter%
  \providecommand\color[2][]{%
    \errmessage{(Inkscape) Color is used for the text in Inkscape, but the package 'color.sty' is not loaded}%
    \renewcommand\color[2][]{}%
  }%
  \providecommand\transparent[1]{%
    \errmessage{(Inkscape) Transparency is used (non-zero) for the text in Inkscape, but the package 'transparent.sty' is not loaded}%
    \renewcommand\transparent[1]{}%
  }%
  \providecommand\rotatebox[2]{#2}%
  \ifx\svgwidth\undefined%
    \setlength{\unitlength}{2943.66616211bp}%
    \ifx\svgscale\undefined%
      \relax%
    \else%
      \setlength{\unitlength}{\unitlength * \real{\svgscale}}%
    \fi%
  \else%
    \setlength{\unitlength}{\svgwidth}%
  \fi%
  \global\let\svgwidth\undefined%
  \global\let\svgscale\undefined%
  \makeatother%
  \begin{picture}(1,0.48624693)%
    \put(0,0){\includegraphics[width=\unitlength]{relation.pdf}}%
    \put(0.0649846,0.3783388){\color[rgb]{0,0,0}\makebox(0,0)[lb]{\smash{$1$}}}%
    \put(0.14441507,0.18561937){\color[rgb]{0,0,0}\makebox(0,0)[lb]{\smash{$2$}}}%
    \put(0.21542658,0.3783388){\color[rgb]{0,0,0}\makebox(0,0)[lb]{\smash{$3$}}}%
    \put(0.28973251,0.18561937){\color[rgb]{0,0,0}\makebox(0,0)[lb]{\smash{$4$}}}%
    \put(0.36586857,0.3783388){\color[rgb]{0,0,0}\makebox(0,0)[lb]{\smash{$5$}}}%
    \put(0.43504988,0.18561937){\color[rgb]{0,0,0}\makebox(0,0)[lb]{\smash{$6$}}}%
    \put(-0.00090236,0.18561937){\color[rgb]{0,0,0}\makebox(0,0)[lb]{\smash{$6$}}}%
    \put(0.61007749,0.3783388){\color[rgb]{0,0,0}\makebox(0,0)[lb]{\smash{$1$}}}%
    \put(0.68950793,0.18561937){\color[rgb]{0,0,0}\makebox(0,0)[lb]{\smash{$2$}}}%
    \put(0.76051944,0.3783388){\color[rgb]{0,0,0}\makebox(0,0)[lb]{\smash{$3$}}}%
    \put(0.83482536,0.18561937){\color[rgb]{0,0,0}\makebox(0,0)[lb]{\smash{$4$}}}%
    \put(0.91096142,0.3783388){\color[rgb]{0,0,0}\makebox(0,0)[lb]{\smash{$5$}}}%
    \put(0.98014275,0.18561937){\color[rgb]{0,0,0}\makebox(0,0)[lb]{\smash{$6$}}}%
    \put(0.54419053,0.18561937){\color[rgb]{0,0,0}\makebox(0,0)[lb]{\smash{$6$}}}%
    \put(0.221963,0.00143847){\color[rgb]{0,0,0}\makebox(0,0)[lb]{\smash{(a)}}}%
    \put(0.76772685,0.00143316){\color[rgb]{0,0,0}\makebox(0,0)[lb]{\smash{(b)}}}%
  \end{picture}%
\endgroup%

%% file: tubes.pdf_tex
%% Creator: Inkscape inkscape 0.48.3.1, www.inkscape.org
%% Accompanies image file 'tubes.pdf' (pdf, eps, ps)
\begingroup%
  \makeatletter%
  \providecommand\color[2][]{%
    \errmessage{(Inkscape) Color is used for the text in Inkscape, but the package 'color.sty' is not loaded}%
    \renewcommand\color[2][]{}%
  }%
  \providecommand\transparent[1]{%
    \errmessage{(Inkscape) Transparency is used (non-zero) for the text in Inkscape, but the package 'transparent.sty' is not loaded}%
    \renewcommand\transparent[1]{}%
  }%
  \providecommand\rotatebox[2]{#2}%
  \ifx\svgwidth\undefined%
    \setlength{\unitlength}{1289.525bp}%
    \ifx\svgscale\undefined%
      \relax%
    \else%
      \setlength{\unitlength}{\unitlength * \real{\svgscale}}%
    \fi%
  \else%
    \setlength{\unitlength}{\svgwidth}%
  \fi%
  \global\let\svgwidth\undefined%
  \global\let\svgscale\undefined%
  \makeatother%
  \begin{picture}(1,1.07935823)%
    \put(0,0){\includegraphics[width=\unitlength]{tubes.pdf}}%
    \put(0.18282479,1.06050439){\color[rgb]{0,0,0}\makebox(0,0)[lb]{\smash{$A_1$}}}%
    \put(0.67735906,0.00585245){\color[rgb]{0,0,0}\makebox(0,0)[lb]{\smash{$A_2$}}}%
    \put(0.15623694,0.77335547){\color[rgb]{0,0,0}\makebox(0,0)[lb]{\smash{$1$}}}%
    \put(0.64722615,0.33199694){\color[rgb]{0,0,0}\makebox(0,0)[lb]{\smash{$2$}}}%
  \end{picture}%
\endgroup%

%% file: constants.pdf_tex
%% Creator: Inkscape inkscape 0.48.3.1, www.inkscape.org
%% Accompanies image file 'constants.pdf' (pdf, eps, ps)
\begingroup%
  \makeatletter%
  \providecommand\color[2][]{%
    \errmessage{(Inkscape) Color is used for the text in Inkscape, but the package 'color.sty' is not loaded}%
    \renewcommand\color[2][]{}%
  }%
  \providecommand\transparent[1]{%
    \errmessage{(Inkscape) Transparency is used (non-zero) for the text in Inkscape, but the package 'transparent.sty' is not loaded}%
    \renewcommand\transparent[1]{}%
  }%
  \providecommand\rotatebox[2]{#2}%
  \ifx\svgwidth\undefined%
    \setlength{\unitlength}{1629.65bp}%
    \ifx\svgscale\undefined%
      \relax%
    \else%
      \setlength{\unitlength}{\unitlength * \real{\svgscale}}%
    \fi%
  \else%
    \setlength{\unitlength}{\svgwidth}%
  \fi%
  \global\let\svgwidth\undefined%
  \global\let\svgscale\undefined%
  \makeatother%
  \begin{picture}(1,0.50850757)%
    \put(0,0){\includegraphics[width=\unitlength]{constants.pdf}}%
    \put(0.24778183,0.49419276){\color[rgb]{0,0,0}\makebox(0,0)[lb]{\smash{A}}}%
    \put(0.78333969,0.48824212){\color[rgb]{0,0,0}\makebox(0,0)[lb]{\smash{A}}}%
  \end{picture}%
\endgroup%

%% file: manifold.pdf_tex
%% Creator: Inkscape inkscape 0.48.3.1, www.inkscape.org
%% Accompanies image file 'manifold.pdf' (pdf, eps, ps)
\begingroup%
  \makeatletter%
  \providecommand\color[2][]{%
    \errmessage{(Inkscape) Color is used for the text in Inkscape, but the package 'color.sty' is not loaded}%
    \renewcommand\color[2][]{}%
  }%
  \providecommand\transparent[1]{%
    \errmessage{(Inkscape) Transparency is used (non-zero) for the text in Inkscape, but the package 'transparent.sty' is not loaded}%
    \renewcommand\transparent[1]{}%
  }%
  \providecommand\rotatebox[2]{#2}%
  \ifx\svgwidth\undefined%
    \setlength{\unitlength}{1231.42841605bp}%
    \ifx\svgscale\undefined%
      \relax%
    \else%
      \setlength{\unitlength}{\unitlength * \real{\svgscale}}%
    \fi%
  \else%
    \setlength{\unitlength}{\svgwidth}%
  \fi%
  \global\let\svgwidth\undefined%
  \global\let\svgscale\undefined%
  \makeatother%
  \begin{picture}(1,1.14466361)%
    \put(0,0){\includegraphics[width=\unitlength]{manifold.pdf}}%
    \put(0.47438153,0.78760898){\color[rgb]{0,0,0}\makebox(0,0)[lb]{\smash{$v$}}}%
    \put(0.46781908,0.19698601){\color[rgb]{0,0,0}\makebox(0,0)[lb]{\smash{$C$}}}%
    \put(0.33394454,1.12098278){\color[rgb]{0,0,0}\makebox(0,0)[lb]{\smash{$T_1$}}}%
    \put(0.59906858,1.12492027){\color[rgb]{0,0,0}\makebox(0,0)[lb]{\smash{$T_2$}}}%
    \put(0.28669469,0.59992219){\color[rgb]{0,0,0}\makebox(0,0)[lb]{\smash{$T_3$}}}%
  \end{picture}%
\endgroup%

%% file: badlocal.pdf_tex
%% Creator: Inkscape inkscape 0.48.3.1, www.inkscape.org
%% Accompanies image file 'badlocal.pdf' (pdf, eps, ps)
\begingroup%
  \makeatletter%
  \providecommand\color[2][]{%
    \errmessage{(Inkscape) Color is used for the text in Inkscape, but the package 'color.sty' is not loaded}%
    \renewcommand\color[2][]{}%
  }%
  \providecommand\transparent[1]{%
    \errmessage{(Inkscape) Transparency is used (non-zero) for the text in Inkscape, but the package 'transparent.sty' is not loaded}%
    \renewcommand\transparent[1]{}%
  }%
  \providecommand\rotatebox[2]{#2}%
  \ifx\svgwidth\undefined%
    \setlength{\unitlength}{4087.475bp}%
    \ifx\svgscale\undefined%
      \relax%
    \else%
      \setlength{\unitlength}{\unitlength * \real{\svgscale}}%
    \fi%
  \else%
    \setlength{\unitlength}{\svgwidth}%
  \fi%
  \global\let\svgwidth\undefined%
  \global\let\svgscale\undefined%
  \makeatother%
  \begin{picture}(1,0.32558609)%
    \put(0,0){\includegraphics[width=\unitlength]{badlocal.pdf}}%
  \end{picture}%
\endgroup%

%% file: blocktograph.pdf_tex
%% Creator: Inkscape inkscape 0.48.3.1, www.inkscape.org
%% Accompanies image file 'blocktograph.pdf' (pdf, eps, ps)
\begingroup%
  \makeatletter%
  \providecommand\color[2][]{%
    \errmessage{(Inkscape) Color is used for the text in Inkscape, but the package 'color.sty' is not loaded}%
    \renewcommand\color[2][]{}%
  }%
  \providecommand\transparent[1]{%
    \errmessage{(Inkscape) Transparency is used (non-zero) for the text in Inkscape, but the package 'transparent.sty' is not loaded}%
    \renewcommand\transparent[1]{}%
  }%
  \providecommand\rotatebox[2]{#2}%
  \ifx\svgwidth\undefined%
    \setlength{\unitlength}{3703.12998047bp}%
    \ifx\svgscale\undefined%
      \relax%
    \else%
      \setlength{\unitlength}{\unitlength * \real{\svgscale}}%
    \fi%
  \else%
    \setlength{\unitlength}{\svgwidth}%
  \fi%
  \global\let\svgwidth\undefined%
  \global\let\svgscale\undefined%
  \makeatother%
  \begin{picture}(1,0.35937842)%
    \put(0,0){\includegraphics[width=\unitlength]{blocktograph.pdf}}%
    \put(0.57147326,0.24631038){\color[rgb]{0,0,0}\makebox(0,0)[lb]{\smash{$s_1$}}}%
    \put(0.97650214,0.2480562){\color[rgb]{0,0,0}\makebox(0,0)[lb]{\smash{$t_1$}}}%
    \put(0.57147326,0.09355165){\color[rgb]{0,0,0}\makebox(0,0)[lb]{\smash{$s_2$}}}%
    \put(0.97650214,0.09529747){\color[rgb]{0,0,0}\makebox(0,0)[lb]{\smash{$t_2$}}}%
    \put(-0.0007173,0.23627195){\color[rgb]{0,0,0}\makebox(0,0)[lb]{\smash{$x_1$}}}%
    \put(0.40431156,0.23627195){\color[rgb]{0,0,0}\makebox(0,0)[lb]{\smash{$x_1$}}}%
    \put(-0.0007173,0.12279405){\color[rgb]{0,0,0}\makebox(0,0)[lb]{\smash{$x_2$}}}%
    \put(0.40431156,0.12279405){\color[rgb]{0,0,0}\makebox(0,0)[lb]{\smash{$x_2$}}}%
  \end{picture}%
\endgroup%